\documentclass[12pt,reqno]{amsart}
\usepackage{amsmath, amssymb, enumitem, float,  mathrsfs,amsaddr,stmaryrd}
\usepackage{tikz,tikz-cd}
\usepackage[headings]{fullpage}
\usepackage{hyperref}

\newenvironment{salign}{
    \begin{equation}
    \begin{aligned}
}{
    \end{aligned}
    \end{equation}
    \ignorespacesafterend
}

\numberwithin{equation}{section}

\theoremstyle{plain}

\newtheorem{theorem}[subsubsection]{Theorem}
\newtheorem{lemma}[subsubsection]{Lemma}
\newtheorem{proposition}[subsubsection]{{Proposition}}
\newtheorem{corollary}[subsubsection]{{Corollary}}

\theoremstyle{definition}
\newtheorem{definition}[subsubsection]{{Definition}}
\newtheorem{example}[subsubsection]{{Example}}

\theoremstyle{remark}
\newtheorem{remark}[subsubsection]{{Remark}}

\newtheoremstyle{RepTheorem} 
	{\topsep}{\topsep}
	{\itshape}
	{}
	{\bfseries}
	{.}
	{ }
	{\thmname{#1}\thmnote{ \bfseries #3}}
\theoremstyle{RepTheorem}
\newtheorem{reptheorem}[subsubsection]{Theorem}

\newcommand{\defeq}{\stackrel{\textnormal{def}}{=}}

\newcommand{\Mod}[1]{\ (\mathrm{mod}\ #1)}

\def\Z {{\mathbb Z}}
\def\F {{\mathbb F}}
\def\Q {{\mathbb Q}}
\def\R {{\mathbb R}}
\def\C {{\mathbb C}}
\def\A {{\mathbb A}}

\def\P {{\mathbb P}}
\def\G {{\mathbb G}}

\def\O {{\mathcal O}}

\def\mcA {{\mathcal A}}
\def\mcB {{\mathcal B}}
\def\mcC {{\mathcal C}}
\def\mcD {{\mathcal D}}
\def\mcE {{\mathcal E}}
\def\mcF {{\mathcal F}}
\def\mcG {{\mathcal G}}
\def\mcI {{\mathcal I}}
\def\mcJ {{\mathcal J}}
\def\mcM {{\mathcal M}}
\def\mcP {{\mathcal P}}
\def\mcS {{\mathcal S}}
\def\mcX {{\mathcal X}}
\def\mcY {{\mathcal Y}}

\def \a {{\mathfrak a}}
\def \b {{\mathfrak b}}
\def \c {{\mathfrak c}}
\def \d {{\mathfrak d}}

\def \n {{\mathfrak n}}
\def \p {{\mathfrak p}}
\def \q {{\mathfrak q}}
\def \g {{\mathfrak g}}

\def \t {{\mathfrak t}}

\def \mfA {{\mathfrak A}}
\def \mfD {{\mathfrak D}}

\def \mfI {{\mathfrak I}}

\def \w {{\textnormal{\textbf{w}}}}

\def \Cl {{\textnormal{Cl}}}

\newcommand{\aff}{\textnormal{aff}}

\newcommand{\GL}{\textnormal{GL}}

\newcommand{\SL}{\textnormal{SL}}
\newcommand{\pr}{\textnormal{pr}}

\newcommand{\Aut}{\textnormal{Aut}}
\newcommand{\Br}{\textnormal{Br}}

\newcommand{\Gal}{\textnormal{Gal}}
\newcommand{\Hom}{\textnormal{Hom}}
\newcommand{\Ht}{\textnormal{Ht}}

\newcommand{\Pic}{\textnormal{Pic}}

\newcommand{\Spec}{\textnormal{Spec}}

\newcommand{\Val}{\textnormal{Val}}
\newcommand{\lcm}{\textnormal{lcm}}

\begin{document}

\title{Points of bounded height in images of morphisms of weighted projective stacks \\
\textnormal{\small With Applications to Counting Elliptic Curves}}

\author{Tristan Phillips}
\email{tristanphillips72@gmail.com}


\begin{abstract}
Asymptotics are given for the number of rational points in the domain of a morphism of weighted projective stacks whose images have bounded height and satisfy a (possibly infinite) set of local conditions. 
As a consequence we obtain results for counting elliptic curves over number fields with prescribed level structures, including the cases of $\Gamma(N)$ for $N\in\{1,2,3,4,5\}$, $\Gamma_1(N)$ for $N\in\{1,2,\dots,10,12\}$, and $\Gamma_0(N)$ for $N\in\{1,2,4,6,8,9,12,16,18\}$. In all cases we give an asymptotic with an expression for the leading coefficient, and in many cases we also give a power-saving error term.
\end{abstract}

\maketitle

\section{Introduction}


\subsection{Counting points on weighted projective stacks}

In 1979, Schanuel \cite{Sch79} proved an asymptotic for the number of rational points of bounded height on projective spaces over number fields, which later motivated Manin's Conjecture on rational points of Fano varieties \cite{FMT89}. This was generalized to weighted projective spaces by Deng in 1998 \cite{Den98}. Bruin and Najman noted that Deng's argument could be slightly modified to address the case of weighted projective stacks \cite{BN22}. In their doctoral thesis, Darda \cite{Dar21}, gives an alternative proof using height zeta functions and applies it to a general class of height functions. In a different direction, Bright, Browning, and Loughran \cite{BBL16} proved a generalization of Schanuel's theorem which allows one to impose infinitely many local conditions on the points of projective space being counted. Recently, Peter Bruin and Manterola Ayala \cite{BMA23, MA21} have generalized Deng's result to count points of weighted projective stacks whose images, with respect to a morphism of weighted projective stacks, are of bounded height in the codomain. In this article we give a common generalization of these results. 
Before stating this generalization, we introduce some notation.

For any set $S$, let $\# S$ denote the cardinality of $S$. Let $K$ be a degree $d$ number field over $\Q$.
 Let $\Val(K)$ denote the set of equivalence classes of non-trivial valuations of $K$ (i.e., \emph{places} of $K$).  
 Let $\Val_0(K)$ denote the set of finite (non-archimedean) places of $K$ and let $\Val_\infty(K)$ denote the set of infinite (archimedean) places of $K$.
 For each place $v\in \Val(K)$ let $K_v$ denote the completion of $K$ at $v$ and let $m_v$ denote the usual Haar measure on $K_v$, normalized so that $m_v$ agrees with Lebesgue measure at all real places and is twice as large as Lebesgue measure at all complex places. 
 Let $K_\infty\defeq \prod_{v\in \Val_\infty(K)} K_v$ and let $m_\infty$ denote the \emph{usual Haar measure} on $K_\infty$ obtained by taking the product of the Haar measures at the infinite places.
 
Let $\w'=(w'_0,\dots,w'_n)$ and $\w=(w_0,\dots,w_n)$ be $(n+1)$-tuples of positive integers. Let $|\w'|\defeq \sum w'_i$ and $\w'_{\min}\defeq \min\{w'_0,\dots,w'_n\}$.
Let $\mcP(\w')$ and $\mcP(\w)$ be the weighted projective stacks with weights $\w'$ and $\w$. 
 For a place $v\in \Val(K)$ and a subset $\Omega\subseteq \mcP(\w')(K_v)$, let $\Omega^{\aff}\subset K_v^{n+1}$ denote the subset of the affine cone of $\mcP(\w')(K_v)$ above $\Omega$ and let $\partial\Omega^{\aff}$ denote the boundary of the affine cone $\Omega^{\aff}$. 
 Let $\Ht_\w$ denote the height function on $\mcP(\w)$ defined in Subsection \ref{subsec:WPSheights}.

We now state our result for counting rational points in the domain of a morphism of weighted projective stacks whose image is of bounded height and where we impose finitely many local conditions.

\begin{theorem}\label{thm:WProjFin}
Let $\varphi:\mcP(\w')\to \mcP(\w)$ be a morphism of weighted projective stacks that has \emph{finite defect} (see Subsection \ref{subsec:defects} for the definition of finite defect). For each infinite place $v\in \Val_\infty(K)$, let $\Omega_v\subset \mcP(\w')(K_v)$ be such that $\Omega_v^{\aff}$ is bounded and definable in an o-minimal structure (see Subsection \ref{subsec:GSo-minimal} for the definition of o-minimal structure). Let $S\subset \Val_0(K)$ be a finite set of finite places. For each place $v\in S$, let $\Omega_v\subset\mcP(\w')(K_v)$ be a subset for which $\Omega_v^{\aff}$ is measurable and $m_v(\partial \Omega_v^{\aff})=0$.
Then we have that
\begin{align*}
&\#\{x\in \mcP(\w')(K):\Ht_\w(\varphi(x))\leq B,\ x\in \varphi(\Omega_v) \text{ for all } v\in S\cup\Val_\infty(K)\}\\
&\hspace{15mm} =\kappa B^{|\w'|/e(\varphi)}+ \begin{cases}
O\left(B^{\frac{1}{e(\varphi)}}\log(B)\right) &  {\substack{\text{ if } \w'=(1,1)\\ \text{and}\ K=\Q,}}\\
O\left(B^{\frac{d|\w'|-w'_{\min}}{d e(\varphi)}}\right) & \text{ else }
\end{cases}
\end{align*}
as a function of $B$, where $\kappa$ is an explicit constant (which we will give an expression for in Section \ref{sec:WPROJ}), and where $e(\varphi)\in \Z_{\geq 1}$ is the reduced degree of $\varphi$ (see Definition \ref{def:ReducedDegree}).
\end{theorem}

Combining this theorem with the geometric sieve, one obtains a version of this result in which infinitely many \emph{admissible} local conditions may be imposed (see Theorem \ref{thm:WProjInfty}). In this case the condition that $\Omega_v$ is \emph{defininable in an o-minimal structure} for $v\in \Val_\infty(K)$ can be relaxed to the condition that $\Omega_v$ is \emph{measurable} for $v\in \Val_\infty(K)$. This generalizes (and gives an alternative proof of) a case of a result of Darda, who proved this in the case $\varphi$ is the identity morphism \cite[Theorem 8.3.2.2]{Dar21}. Darda uses height zeta functions to count points on weighted projective stacks for a class of height functions he calls \textit{quasi-toric}. Although the height we are using can be viewed as the pullback of a quasi-toric height, it seems unlikely that the leading coefficient in Theorem \ref{thm:WProjFin} would exactly agree with that in Darda since the pullback height only equals the height of the image, $\Ht_{\w}(\phi(x))$, up to a constant. Darda also does not obtain error terms in any cases. 

The proof of Theorem \ref{thm:WProjFin} and Theorem \ref{thm:WProjInfty} make use of geometry-of-numbers techniques. The general ideas are similar to those in the papers cited prior to the theorem, but some care is required to achieve the generalizations simultaneously. We also streamline part of Schanuel's proof by using o-minimal geometry via a lattice counting theorem of Barroero and Widmer \cite[Theorem 1.3]{BW14}. In order to deal with infinitely many local conditions, we prove a weighted version of the geometric sieve for integral points on affine spaces over number fields (Proposition \ref{prop:KEk}), which may be of independent interest. 

A stacky version of the Batyrev--Manin Conjecture has recently been formulated \cite{ESZB23,DY24,DY25+}. It would be interesting to verify to what extent Theorem \ref{thm:WProjFin} agrees with what those conjectures predict, although we do not pursue that here.


Motivated by counting twists of elliptic curves, we use our results for counting elliptic curves satisfying infinite sets of admissible local conditions (Theorem \ref{thm:WProjInfty}) in order to prove a result for counting points on weighted projective stacks with respect to a certain \emph{twisted height} function (see Theorem \ref{thm:WProjTwisted}).

\subsection{Counting elliptic curves with prescribed level structure}

Our results for counting points on weighted projective stacks have immediate consequences for counting elliptic curves \textit{equipped} with a prescribed level structure. With some additional work we are able to count elliptic curves which \textit{admit} a prescribed level structure.

Let $N$ be a positive integer and let $G$ be a subgroup of $\GL_2(\Z/N\Z)$. An elliptic curve $E$ over a number field $K$ is said to \textbf{admit a $G$-level structure} if there is a $(\Z/N\Z)$-basis for the (geometric) $N$-torsion subgroup, $E[N](\overline{K})$, such that the image of the mod-$N$ Galois representation,
\[
\rho_{E,N}: \Gal(\overline{K}/K)\to \Aut(E[N](\overline{K}))\cong \GL_2(\Z/N\Z),
\]
is contained in $G$. In the identification $\Aut(E[N](\overline{K}))\cong \GL_2(\Z/N\Z)$ we view geometric $N$-torsion points $E[N](\overline{K})\cong \Z/N\Z \oplus \Z/N\Z$ as row vectors which are acted upon on the right by elements of $\GL_2(\Z/N\Z)$. 

For example, if $M|N$ are positive integers and
\[
G(M,N)=\left\{g\in \GL_2(\Z/N\Z) : g=
\begin{pmatrix}
\ast & \ast\\
0 & 1
\end{pmatrix}
\text{ and }
g\equiv
\begin{pmatrix}
\ast & 0\\
0 & 1
\end{pmatrix}
\Mod{M}\right\},
\]
then, when $K$ contains the cyclotomic extension $\Q(\zeta_N)$, elliptic curves over $K$ which admit a $G(M,N)$-level structure are those whose torsion subgroup contains $\Z/M\Z\times \Z/N\Z$.

 As another example, if $N$ is a positive integer and
\[
G_0(N)=\left\{g\in \GL_2(\Z/N\Z) : g=
\begin{pmatrix}
\ast & \ast\\
0 & \ast
\end{pmatrix}\right\},
\]
then elliptic curves which admit a $G_0(N)$-level structure are those which have a rational cyclic $N$-isogeny. 

For any level structure $G$, let $\Gamma_G$ denote the inverse image of $G$ with respect to the canonical map \[
\SL_2(\Z)\to \SL_2(\Z/N\Z) \hookrightarrow \GL_2(\Z/N\Z).
\] 
 Let $\mcY_G$ be the moduli stack parameterizing isomorphism classes of elliptic curves with $G$-level structure, and let $\mcX_G$ be the Deligne--Mumford compactification of $\mcY_G$. Let $K_G$ denote the subfield of the cyclotomic field $\Q(\zeta_N)$ fixed by the action of $G$ given by
\begin{align*}
G\times \Q(\zeta_N)&\to \Q(\zeta_N)\\
(g, \zeta_N)&\mapsto \zeta_N^{\det(g)}.
\end{align*} 
We will view $\mcY_G$ and $\mcX_G$ as algebraic stacks over $K_G$.

There are many examples of modular curves $\mcX_G$ which are isomorphic to weighted projective stacks $\mcP(w_0,w_1)$ (see Table \ref{tab:EllipticCount} for some examples).
As the compactified moduli stack of elliptic curves, $\mcX_{\GL_2(\Z)}$, is isomorphic to the weighted projective stack $\mcP(4,6)$, the height $\Ht_{(4,6)}$ on $\mcP(4,6)$ induces a height on elliptic curves, and the 12-th power of this height coincides with the usual \textbf{naive height} of an elliptic curve, defined as follows: If $E$ is an (isomorphism class of) elliptic curve(s) over $K$ with reduced integral short Weierstrass equation $y^2=x^3+Ax+B$, then the naive height of $E$ is defined to be
\[
\Ht(E)\defeq \prod_{v\in \Val_\infty(K)} \max\{|A|_v^3, |B|_v^2\}=\Ht_{(4,6)}([A:B])^{12}.
\]

As a consequence of Theorem \ref{thm:WProjFin}, we obtain the following result for counting elliptic curves admitting prescribed level structures.

\begin{theorem}\label{thm:EllipticCount}
 Let $G$ be a subgroup of $\GL_2(\Z/N\Z)$ and let $K$ be a degree $d$ number field containing $K_G$. Suppose that for some pair of positive integers $\w=(w_0,w_1)$, there is an isomorphism of $K_G$-stacks from the modular curve $\mcX_G$ to the weighted projective stack $\mcP(\w)$. Suppose that the morphism that forgets the level structure, $\varphi_G:\mcX_G\to \mcX_{\GL_2(\Z)}$, satisfies $(w_0, w_1)=1$ or $e(\varphi_G))=1$. 
Then the number of isomorphism classes of elliptic curves over $K$ which admit a $G$-level structure and have naive  height less than $B$ is
\[
\kappa_G B^{\frac{|\w|}{12 e(\varphi_G)}}+O\left(B^{\frac{d|\w|-w_{\min}}{12 e(\varphi_G) d}}\right),
\]
where $\kappa_G$ is an explicit constant which we give an expression for in Section \ref{sec:EC}.
\end{theorem}


One can show that under the assumptions of Theorem \ref{thm:EllipticCount},
\[
e(\varphi_G)=\frac{w_0w_1}{24}[\SL_2(\Z):\Gamma_G],
\]
 and that $\varphi_G$ has finite defect if $(w_0,w_1)=(1,1)$ or $e(\varphi_G)=1$ (see Proposition \ref{prop:finite-defect}).

Table \ref{tab:EllipticCount} displays all level structures of the form $G(M,N)$ or $G_0(N)$ in which Theorem \ref{thm:EllipticCount} applies (see also \cite[Table 1]{BN22}, which contains some additional cases where the Theorem applies). 
 In particular, the theorem applies to the torsion subgroups in Mazur's classification of the possible torsion subgroups for elliptic curves over $\Q$, which correspond to the first 15 rows of Table \ref{tab:EllipticCount}.
Note that for all $M|N$ we have $K_{G(M,N)}=\Q$ and $K_{G_0(N)}=\Q$. In particular, for $G(M,N)$-level structures and $G_0(N)$-level structures, Theorem \ref{thm:EllipticCount} applies for $K$ an arbitrary number field.

\begin{table}
\centering
\begin{tabular}{cccccc}
\hline
$G$ & $\Gamma_G$ & $(w_0,w_1)$ & $e(\varphi_G)$ & $\frac{w_0+w_1}{12 e(G)}$ & $\frac{d(w_0+w_1)-w_{\min}}{12 d e(G)}$\\
\hline
$G(1,1)=G_0(1)$ & $\Gamma(1)=\SL_2(\Z)$ & $(4,6)$ & $1$ & $5/6$ & $5/6-1/3d$\\
$G(1,2)=G_0(2)$ & $\Gamma_1(2)=\Gamma_0(2)$ & $(2,4)$ & $1$ & $1/2$ & $1/2-1/6d$\\
$G(1,3)$ & $\Gamma_1(3)$ & $(1,3)$ & $1$ & $1/3$ & $1/3-1/12d$\\
$G(1,4)$ & $\Gamma_1(4)$ & $(1,2)$ & $1$ & $1/4$ & $1/4-1/12d$\\
$G(1,5)$ & $\Gamma_1(5)$ & $(1,1)$ & $1$ & $1/6$ & $1/6-1/12d$\\
$G(1,6)$ & $\Gamma_1(6)$ & $(1,1)$ & $1$ & $1/6$ & $1/6-1/12d$\\
$G(1,7)$ & $\Gamma_1(7)$ & $(1,1)$ & $2$ & $1/12$ & $1/12-1/24d$\\
$G(1,8)$ & $\Gamma_1(8)$ & $(1,1)$ & $2$ & $1/12$ & $1/12-1/24d$\\
$G(1,9)$ & $\Gamma_1(9)$ & $(1,1)$ & $3$ & $1/18$ & $1/18-1/36d$\\
$G(1,10)$ & $\Gamma_1(10)$ & $(1,1)$ & $3$ & $1/18$ & $1/18-1/36d$\\
$G(1,12)$ & $\Gamma_1(12)$ & $(1,1)$ & $4$ & $1/24$ & $1/24-1/48d$\\
$G(2,2)$ & $\Gamma(2)$ & $(2,2)$ & $1$ & $1/3$ & $1/3-1/6d$\\
$G(2,4)$ & $\Gamma(2,4)$ & $(1,1)$ & $1$ & $1/6$ & $1/6-1/12d$\\
$G(2,6)$ & $\Gamma(2,4)$ & $(1,1)$ & $2$ & $1/12$ & $1/12-1/24d$\\
$G(2,8)$ & $\Gamma(2,8)$ & $(1,1)$ & $4$ & $1/24$ & $1/24-1/48d$\\
$G(3,3)$ & $\Gamma(3)$ & $(1,1)$ & $1$ & $1/6$ & $1/6-1/12d$\\
$G(3,6)$ & $\Gamma(3,6)$ & $(1,1)$ & $3$ & $1/18$ & $1/18-1/36d$\\
$G(4,4)$ & $\Gamma(4)$ & $(1,1)$ & $2$ & $1/12$ & $1/12-1/24d$\\
$G(5,5)$ & $\Gamma(5)$ & $(1,1)$ & $5$ & $1/30$ & $1/30-1/60d$\\
$G_0(4)$ & $\Gamma_0(4)$ & $(2,2)$ & $1$ & $1/3$ & $1/3-1/6d$\\
\hline
\end{tabular}
\caption{Cases of Theorem \ref{thm:EllipticCount}}
\label{tab:EllipticCount}
\end{table}

Table \ref{tab:ModularCurvesWPS} gives lower bounds for the number of modular curves over $\Q$ isomorphic to various weighted projective stacks. Determining if a modular curve $\mcX_G$ is isomorphic to a weighted projective stack over $\Q$ amounts to determining the following:
\begin{itemize}
	\item the genus of $\mcX_G$,
	\item if $\mcX_G$ has a $\Q$ point,
	\item the number of order $2$ elliptic points of $\mcX_G$,
	\item the number of order $3$ elliptic points of $\mcX_G$, and
	\item the number of irregular cusps of $\mcX_G$.
\end{itemize}
Based on data from the L-functions and modular forms database (LMFDB) \cite[\href{https://beta.lmfdb.org/ModularCurve/Q/}{Modular curves over $\Q$}]{lmfdb}, there are $11567$ genus zero modular curves of level $\leq 50$ over $\Q$, and of those, at least $7105$ have a $\Q$-point. For these modular curves, the LMFDB gives the elliptic points, and the irregular cusps can be found using a GitHub repository of Zywina \cite{ZywinaGithub,Zyw25+} . When $\mcX_G\cong \mcP(1,1)$, we have $(w_0,w_1)=(1,1)$, so Theorem \ref{thm:EllipticCount} applies. In order to determine how often Theorem \ref{thm:EllipticCount} applies  when $(w_0,w_1)\neq (1,1)$, we compute $e(\varphi_G)$, which amounts to computing the index of $G$ in $\GL_2$. For example, when $\mcX_G\cong \mcP(2,2)$, in order for Theorem \ref{thm:EllipticCount} to apply, we must have that the index of $G$ in $\GL_2$ equals $6$.  

\begin{table}
\centering
\begin{tabular}{ccccccc}
\hline
 & $\mcP(1,1)$ & $\mcP(1,2)$ & $\mcP(1,3)$ & $\mcP(2,2)$ & $\mcP(2,6)$ & $\mcP(4,6)$\\
\hline
Total & $2568$ & $44$ & $16$ & $398$ & $1$ & $1$\\
Theorem \ref{thm:EllipticCount} applies& $2568$ & $44$ & $16$ & $38$ & $0$ & $1$\\
Theorem \ref{thm:EllipticTwistCount} applies&  &  &  & $398$  &  & \\
\hline
\end{tabular}
\caption{Modular curves over $\Q$ of level $\leq 50$ isomorphic to a weighted projective stack.}
\label{tab:ModularCurvesWPS}
\end{table}

 When the morphism $\varphi_G$ does not have finite defect, the problem of counting elliptic curves which admit a $G$-level structure becomes more difficult. Nevertheless, we are able to prove the following result in the case that $\mcX_G\cong \mcP(2,2)$, which does not assume the morphism $\mcX_G\to \mcX_{\GL_2(\Z)}$ forgetting the level structure has finite defect.
 
\begin{theorem}\label{thm:EllipticTwistCount}
Let $G$ be a subgroup of $\GL_2(\Z/N\Z)$.  Suppose that there is an isomorphism $\mcX_G\cong \mcP(2,2)$ defined over $K_G$. 
Define 
\begin{align*}
e'(G)\defeq \frac{[\SL_2(\Z):\Gamma_{G}]}{12}.
\end{align*}
Let $K$ be a number field of degree $d$ containing $K_G$.
Then there exists explicit, non-zero constants $\kappa_1$ and $\kappa_2$, such that the number of isomorphism classes of elliptic curves over $K$ admitting a $G$-level structure and which have naive height less than $B$ is asymptotic to
\[
\begin{cases}
\kappa_1 B^{1/6}\log(B) & \text{ if } e'(G)=1,\\
\kappa_2 B^{1/6} & \text{ if } e'(G)>1.
\end{cases}
\]
Expressions for the constants $\kappa_1$ and $\kappa_2$ will be given in Section \ref{sec:WPROJ}.
\end{theorem}

Although the condition $\mcX_G\cong \mcP(2,2)$ in Theorem \ref{thm:EllipticTwistCount} may seem restrictive, it turns out that any modular curve isomorphic to a $\mu_2$-gerbe over $\P^1$ must be isomorphic to $\mcP(2,2)$ (see Proposition \ref{prop:mu_2-gerbe_K-rational} and Proposition \ref{prop:mu_2-gerbe_non-trivial}).

In reference to Table \ref{tab:ModularCurvesWPS}, we see that Theorem \ref{thm:EllipticCount} and Theorem \ref{thm:EllipticTwistCount} together give asymptotics for counting elliptic curves over $\Q$ with $G$-level structure by naive height for nearly all $G$ of level bounded by $50$ and $\mcX_G$ isomorphic to a weighted projective stack, with the only exception being when $\mcX_G\cong \mcP(2,6)$. The single case when $\mcX_G\cong \mcP(2,6)$ is $G=G_0(3)$, which is dealt with in \cite{PPV20} over $\Q$ (see Remark \ref{rem:X0(3)} for further discussion).  

Theorem \ref{thm:EllipticTwistCount} allows one to count $K$-rational cyclic $N$-isogenies for $N\in \{6,8,9,12,16,18\}$, and Table \ref{tab:EllipticTwistCount} gives the values for $e'(G)$ in these cases. 

\begin{table}[h]
\centering
\begin{tabular}{ccc}
\hline
$G$ & $\Gamma_G$ &  $e'(G)$\\
\hline
$G_0(6)$ & $\Gamma_0(6)$ & $1$\\
$G_0(8)$ & $\Gamma_0(8)$ & $1$\\
$G_0(9)$ & $\Gamma_0(9)$ & $1$\\
$G_0(12)$ & $\Gamma_0(12)$ & $2$\\
$G_0(16)$ & $\Gamma_0(16)$ & $2$\\
$G_0(18)$ & $\Gamma_0(18)$ & $3$\\
\hline
\end{tabular}
\caption{Cases of Theorem \ref{thm:EllipticTwistCount}}
\label{tab:EllipticTwistCount}
\end{table}

Theorem \ref{thm:EllipticCount} and Theorem \ref{thm:EllipticTwistCount} recover and extend a number of previously known results, which we now briefly survey. 

Over $\Q$, many cases of Theorem \ref{thm:EllipticCount} were previously proven: In 1992, using a basic lattice point counting argument together with M\"obius inversion, Brumer counted isomorphism classes of elliptic curves (i.e., the case of $G(1,1)$-level structure) \cite{Bru92}. In 2000, Grant counted isomorphism classes of elliptic curves which admit a $G(1,2)$-level structure \cite{Gra00}; this was done by exploiting several properties of elliptic curves containing a rational two-torsion point, and ultimately reducing the problem to counting lattice points in a semi-algebraic region.  Harron and Snowden counted isomorphism classes of elliptic curves which admit a $G(1,3)$-level structure \cite{HS17}; they gave a general strategy for reducing such problems to lattice point counting problems which can be addressed using the Principle of Lipschitz \cite{Dav51} (see also \cite[VI. \S 2. Theorem 2]{Lan94}). The method of Harron and Snowden allowed them to give a uniform proof of the $G(1,1)$, $G(1,2)$, and $G(1,3)$ cases. Cullinan, Kenney, and Voight showed that the method of Harron and Snowden could be applied more widely, and doing so were able to count isomorphism classes of elliptic curves for all level structures $G$ for which $\mcX_G\cong \mcP(1,1)$, and also for the cases of $G(2,2)$-level structure and $G_0(4)$-level structure \cite{CKV22}. For other results on counting 4-isogenies see \cite{PS21}.

Over the Gaussian numbers, $\Q(i)$, Zhao \cite{Zha20} has proven results on counting elliptic curves with prescribed torsion, using methods similar to those of Harron and Snowden. Our Theorem \ref{thm:EllipticCount} extends these results (by also giving a leading coefficient and error term) in the cases for which the underlying modular curve has genus zero. However, an interesting feature of Zhao's work is that he also considers cases in which the underlying modular curve has genus one (and positive rank). Such cases lie outside the scope of this paper, but studying points of bounded height on genus one modular curves with positive rank over more general number fields would be an interesting avenue for future research, and could perhaps be related to work in Iwasawa theory on the growth of Mordell--Weil ranks in towers of number fields.

  Over arbitrary number fields, Bruin and Najman \cite{BN22} showed that, if there is a representable morphism between a genus zero modular curve $\mcX_G$ and the weighted projective stack $\mcP(w_0,w_1)$, then the asymptotic growth rate for counting isomorphism classes of elliptic curves with prescribed torsion is $\frac{w_0+w_1}{12e(G)}$ (where $e(G)$ is as in Theorem \ref{thm:EllipticCount}). A key idea of Bruin and Najman is to exploit the fact that many genus zero modular curves are weighted projective stacks. Our results will also exploit this relation with weighted projective stacks, though we take a more direct strategy than Bruin and Najman, which allows us to compute the leading coefficient and, in many cases, a power-saving error term.

For the the level structures listed in Table \ref{tab:EllipticTwistCount}, Boggess and Sankar have given the growth rates for counting isomorphism classes of elliptic curves over $\Q$ \cite{BS24}. What typically makes counting elliptic curves with an $N$-isogeny more difficult than counting elliptic curves with a prescribed torsion subgroup, is that the morphism $\mcX_{G_0(N)}\to\mcX_{\GL_2(\Z)}$ rarely has finite defect (the exceptions being $N\in\{1,2,4\}$). This occurs more frequently because the moduli stack $\mcX_{G_0(N)}$ has generic inertia stack $B\mu_2$ (since $-I\in \Gamma_0(N)$). Boggess and Sankar are able to deal with this for some $N$ by constructing a double cover $\mcX_{G_{1/2}(N)}$ of $\mcX_{G_0(N)}$ for which every elliptic curve that admits a $G_0(N)$-level structure has a unique twist which admits a $G_{1/2}(N)$-level structure \cite[Lemma 2.1]{BS24}.\footnote{The $\mcX_{G_{1/2}(N)}$ are \textit{quadratic refinements} of the $\mcX_{G_0(N)}$, i.e., $G_0(N)=\langle -I,G_{1/2}(N)\rangle$.} Then they count twists of elliptic curves admitting a $G_{1/2}$-level structure by using an elementary (but complicated) counting argument, which generalizes a similar argument appearing in the work of Harron and Snowden. Theorem \ref{thm:EllipticTwistCount} not only extends the results of \cite{BS24} for counting elliptic curves over $\Q$ with $N$-isogenies for $N\in \{6,8,9,12,16,18\}$ to elliptic curves over arbitrary number fields, but also provides a leading coefficient for the asymptotics, which refines their result for counting elliptic curves over $\Q$. 

There are some results for counting elliptic curves over $\Q$ with prescribed level structure which we were not able to generalize to arbitrary number fields, but are worth mentioning. Pomerance, Pizzo, and Voight \cite{PPV20} gave a very precise asymptotic for counting elliptic curves which admit a $G_0(3)$-level structure, which gives several lower order terms. In \cite[\S 5]{BS24}, Boggess and Sankar  develop a method (different from their method mentioned earlier) for counting elliptic curves admitting a $G_0(N)$-level structure, which makes use of the structure of the ring of modular forms. With this method they are able to give an asymptotic growth rate in the cases that $N\in\{2,3,4,5,6,8,9\}$. The case of $G_0(5)$-level structure is particularly interesting since the modular curve $\mcX_{G_0(5)}$ has two (quadratic) elliptic points. 

Since the initial draft of this article, several more relevant results have appeared. In \cite{MV23} Molnar and Voight introduce several new ideas in order to count elliptic curves over $\Q$ with $G_0(7)$-level structure. These ideas were further explored in in Molnar's thesis \cite{Mol23}, where asymptotics (with power-saving error) are given for counting elliptic curves with $G_0(N)$-level structure for $G\in\{6,7,8,9,10,12,13,16,18,25\}$. This improves upon Theorem \ref{thm:EllipticTwistCount} when $K=\Q$ in these cases by providing an error term, and in some cases a secondary term as well. Further work of Arango-Piñeros, Han, Padurariu, and Park find an asymptotic for counting elliptic curves over $\Q$ with $G_0(5)$-level structure which gives three leading terms and an error term \cite{APHPP25+}.

In \cite{CJP24+}, Cho, Jeong, and Park build off the results of this article to conditionally bound the average analytic rank of elliptic curves with prescribed level structure over number fields.


\begin{remark}
By Faltings's Theorem, if the coarse space $X_G$ of $\mcX_G$ contains infinitely many $K$-rational points, then $X_G$ must have genus zero or genus one. In this article we have focus on the genus zero case.
The case in which $X_G$ has genus one is interesting in its own right, but would likely need to be studied using different methods than those presented here. In particular, the number of rational points on a genus one curve is largely dictated by its Mordell-Weil rank. For example, the number of points of bounded height (with respect to the canonical bundle) on an elliptic curve $E$ is $\sim \log(B)^{r/2}$, where $r$ is the rank of $E$. 
\end{remark}

Finally, we mention work of Olechnowicz, which counts points of bounded height in the image of a morphism of projective spaces, $\varphi:\P^n\to\P^m$ \cite{Ole24+}. We suspect that Theorem \ref{thm:WProjFin} can be modified to allow morphisms between weighted projective stacks of different dimensions. Olechnowiz's gives applications of his results to arithmetic dynamics \cite{Ole23+}, see also the work of Siu \cite{Siu24+} for further results in this direction.

\subsection{Organization}

In Section \ref{sec:Weighted-Projective-Spaces} we recall and prove some facts about weighted projective stacks, focusing on morphisms between weighted projective stacks and heights defined on weighted projective stacks. 
In Section \ref{sec:Geometric-Sieve} we prove several results for counting points on affine spaces over number fields satisfying prescribed local conditions. 
In Section \ref{sec:WPROJ} we prove our results for counting points of bounded height on weighted projective stacks. This is where we prove Theorem \ref{thm:WProjFin} and Theorem \ref{thm:WProjTwisted}.
In Section \ref{sec:EC} we apply the results of Section \ref{sec:WPROJ} to count points of bounded height on genus zero modular curves. This is where we prove our results for counting elliptic curves with prescribed level structures, Theorem \ref{thm:EllipticCount} and Theorem \ref{thm:EllipticTwistCount}.

\subsection{Acknowledgements}
The author would like to thank Brandon Alberts, Peter Bruin, Peter Cho, Jordan Ellenberg, Keunyoung Jeong, Robert Lemke Oliver, Irati Manterola Ayala, Grant Molnar, Filip Najman, Matt Olechnowicz, and John Voight for helpful conversations and comments on this work. The author is also grateful for the anonymous referee's thoughtful comments and corrections. Special thanks to Bryden Cais and Doug Ulmer for their encouragement during this project. 
During the completion of this work the author was supported by the National Science Foundation, via grant DMS-2303011.

\section{Preliminaries on weighted projective stacks}\label{sec:Weighted-Projective-Spaces}

In this section we recall facts about weighted projective stacks with an emphasis on morphisms between weighted projective stacks and heights on weighted projective stacks.

\subsection{Weighted projective stacks}

Let $\A^n$ denote affine $n$-space and let $\mathbb{G}_m$ denote the multiplicative group scheme. Given an $(n+1)$-tuple of positive integers $\w=(w_0,\dots,w_n)$, the \textbf{weighted projective stack} $\mcP(\w)$ is defined to be the quotient stack
\[
\mcP(\w)\defeq [(\A^{n+1}-\{0\})/\mathbb{G}_m],
\]
where $\mathbb{G}_m$ acts on $\A^{n+1}-\{0\}$ by the \textbf{$\w$-weighted action}, 
\begin{align*}
\mathbb{G}_m\times(\A^{n+1}-\{0\}) &\to (\A^{n+1}-\{0\})\\
(\lambda,(x_0,\dots,x_n)) &\mapsto \lambda\ast_\w (x_0,\dots,x_n)\defeq (\lambda^{w_0}x_0,\dots, \lambda^{w_n}x_n).
\end{align*}
For example, when $\w=(1,\dots,1)$ we have that $\mcP(\w)$ is the usual projective space $\P^n$. 

The algebraic stack $\mcP(\w)$ is
 smooth (since it is a quotient stack of a smooth scheme by a smooth group scheme) and
proper (by the valuative criterion for stacks). The point $[(a_0,\dots,a_n)]\in\mcP(w_0,\dots,w_n)$ has stabilizer $\mu_m$ where $m=\gcd(w_i : a_i\neq 0)$. When $K$ is a field of characteristic zero we have that $\mu_m$ is finite and reduced over $K$. It follows that, over fields of characteristic zero, $\mcP(w_0,\dots,w_n)$ is a Deligne--Mumford stack. 

For any field $F$ let $\mcP(\w)(F)$ denote the set of isomorphism classes of the groupoid of $F$-points of $\mcP(\w)$. More concretely, $\mcP(\w)(F)$ is in bijection with the quotient
\[
\mcP(\w)(F)=(F^{n+1}-\{0\})/F^\times
\]
where $F^\times$ acts on $F^{n+1}-\{0\}$ by the weighted action
\begin{align*}
F^\times \times (F^{n+1}-\{0\}) &\to (F^{n+1}-\{0\})\\
(\lambda, (a_0,\dots,a_n)) &\mapsto (\lambda^{w_0}a_0,\dots,\lambda^{w_n}a_n).
\end{align*}

Throughout this article, unless otherwise stated, we will assume all weighted projective stacks $\mcP(\w)$ are defined over a number field $K$. 

\subsection{Morphisms of weighted projective stacks}\label{subsec:WPSmorphisms}

\sloppy Let $\w=(w_0,\dots,w_n)$ and $\w'=(w'_0,\dots,w'_n)$ be $(n+1)$-tuples of positive integers and let $\mcP(\w)$ and $\mcP(\w')$ be the corresponding weighted projective stacks over $K$. Let $\Hom_K(\mcP(\w'),\mcP(\w))$ denote the set of morphisms of stacks from $\mcP(\w')$ to $\mcP(\w)$ defined over $K$.

Let $K[x'_0,\dots,x'_n]$ be the graded $K$-algebra where $x'_i$ has weight $w'_i$ for each $i$. For each $i\in\{0,\dots, n\}$ let $f_i\in K[x'_0,\dots,x'_n]$, and let $f=(f_i)_{i=0}^n$. If there exists an $e(f)\in \Z_{\geq 0}$ such that each $f_i$ is a weighted homogeneous polynomial of degree $e(f)w_i$, and if the homogeneous ideal $\sqrt{(f_0,\dots,f_n)}\subseteq K[x'_0,\dots,x'_n]$ contains $(x'_0,\dots,x'_n)$, then (by the same argument as \cite[Lemma 4.1]{BN22}) the graded ring homomorphism
\begin{align*}
K[x_0,\dots,x_n]&\to K[x'_0,\dots,x'_n]\\
x_i &\mapsto f_i
\end{align*}
induces a morphism of weighted projective stacks
\[
\varphi_f:\mcP(\w')\to \mcP(\w).
\]

Let $\Phi_{\w',\w}(K)$ denote the set of $f=(f_i)_{i=0}^n$ with 
\begin{equation*}
f_i\in K[x_0',\dots,x'_n]\qquad \text{ and }\qquad (x_0',\dots, x_n') \subseteq \sqrt{(f_0,\dots, f_n)}.
\end{equation*}
 There is a natural (weighted) action of $\G_m(K)$ on $\Phi_{\w',\w}(K)$, given by
\begin{align*}
\G_m(K)\times \Phi_{\w',\w}(K)&\to \Phi_{\w',\w}(K)\\
(\lambda, (f_i)) & \mapsto \lambda\ast_{\w} (f_i)=(\lambda^{w_i} f_i).
\end{align*} 
The following theorem gives a description of morphisms between weighted projective stacks:
 
\begin{theorem}\label{thm:WProjMorphisms}
Using the notation from above, there is a bijection
\begin{align*}
\Phi_{\w',\w}(K)/\G_m(K) &\to \Hom_K(\mcP(\w'),\mcP(\w))\\
f=(f_i) &\mapsto \varphi_f.
\end{align*}
\end{theorem}

\begin{proof}
This follows from an explicit description of morphisms between quotient stacks due to Bruin and Najman \cite[Lemma A.2]{BN22}. In particular, the proof in the case of weighted projective lines, \cite[Lemma 4.1]{BN22}, works for higher dimensional weighted projective stacks.  
\end{proof}

\begin{definition}[Reduced degree]\label{def:ReducedDegree}
Let $\varphi:\mcP(\w')\to\mcP(\w)$ be a morphism of weighted projective stacks over $K$. Let $f=(f_i)\in\Phi_{\w',\w}(K)/\G_m(K)$ correspond to $\varphi$. Then the \textbf{reduced degree}, denoted $e(\varphi)$ or $e(f)$, is the constant $e(f)$ defined above.
\end{definition}

\begin{remark}
Equivalently, $e(\varphi)$ can be defined as the integer for which the pullback of the hyperplane bundle $\varphi^{\ast} \O_{\P(\w)}(1)$ equals $\O_{\P(\w')}(e(\varphi))$.
\end{remark}

For morphisms of $1$-dimensional weighted projective stacks, the degree and reduced degree are related as follows:
\begin{equation}
	\deg(\varphi) = e(\varphi)\frac{\lcm(w_0,w_1)}{\lcm(w'_0,w'_1)},
	\end{equation}
where $\lcm(w_0,w_1)=w_0\cdot w_1/\gcd(\w)$ accounts for the generic $\mu_{\gcd(\w)}$-gerbe of $\mcP(\w)$, and $\lcm(w'_0,w'_1)$  accounts for the $\mu_{\gcd(\w')}$-gerbe of $\mcP(\w')$.






\subsection{Heights on weighted projective stacks}\label{subsec:WPSheights}

For each place $v\in \Val_0(K)$ let $\p_v$ denote the corresponding prime ideal of the ring of integers $\O_K$ of $K$, and let $\pi_v$ be a uniformizer for the completion $K_v$ of $K$ at $v$. Define $\pi_{\p_v}\defeq \pi_v$ and $K_{\p_v}\defeq K_v$. For $x=(x_0,\dots,x_n)\in K^{n+1}-\{0\}$, define $|x_i|_{\w,v}\defeq |\pi_v|_v^{\lfloor v(x_i)/w_i\rfloor}$. For $v\in \Val_\infty(K)$, let $|\cdot|_v$ denote the usual absolute value. Define
\[
 |x|_{{\w},v}\defeq \begin{cases}
 \max_i\{|x_i|_{\w,v}\} & \text{ if } v\in \Val_0(K),\\
 \max_i\{|x_i|_v^{1/w_i}\} & \text{ if } K_v\cong \R,\\
 \max_i\{|x_i|_v^{2/w_i}\} & \text{ if } K_v\cong \C,
 \end{cases}
 \]
and 
\[
|x|_{\w,\infty}\defeq \prod_{v\in \Val_\infty(K)}|x|_{\w,v}.
\]
 
 \begin{definition}[Height]\label{def:height}
  The \textbf{(exponential) height} of a point $x=[x_0:\cdots:x_n]\in \mcP(\w)(K)$ is defined as
\[
\Ht_\w(x)\defeq \prod_{v\in \Val(K)} |(x_0,\dots,x_n)|_{\w,v}.
\]
\end{definition}
It is straightforward to check that this height function does not depend on the choice of representative of $x$.

\begin{definition}[Scaling ideal]
Define the \textbf{scaling ideal} $\mfI_{\w}(x)$ of $x=(x_0,\dots,x_n)\in K^{n+1}-\{0\}$ to be the fractional ideal
\[
\mfI_{\w}(x)\defeq \prod_{v\in \Val_0(K)} \p_v^{\min_i\{\lfloor v(x_i)/w_i\rfloor\}}.
\]
\end{definition}

Note that different representatives for the same point of $\mcP(\w)(K)$ can have different scaling ideals. 

The scaling ideal $\mfI_{\w}(x)$ can be characterized as the intersection of all fractional ideals $\a$ of $\O_K$ such that $x\in \a^{w_0}\times \cdots\times\a^{w_n}$. It has the property that
\[
\mfI_\w(x)^{-1}=\{a\in K : a^{w_i}x_i\in \O_K \text{ for all } i\}.
\]
Let $N_{K/\Q}$ denote the norm function defined on fractional ideals of $K$.
It is straightforward to check that the height from Definition \ref{def:height} can be written in terms of the scaling ideal as follows:
\[
\Ht_\w([x_0:\dots:x_n])=\frac{|x|_{\w,\infty}}{N_{K/\Q}(\mfI_\w(x))}.
\]

\begin{remark}
The height $\Ht_{\w}$ is known to have the Northcott property, i.e., for all $B\in R_{>0}$ the set
\begin{equation*}
	\{ x\in \mcP(\w)(K) : \Ht_{\w}(x)\leq B\}
\end{equation*}
is finite. In fact Deng \cite{Den98}, who first defined this height, gave an asymptotic for the number of points of bounded height. In the case of projective spaces $\P^n=\mcP(1,\dots,1)$, this height corresponds to the usual Weil height. In more geometric terms, this height on weighted projective stacks can be viewed as the `stacky height' associated to the line bundle $\O_{\mcP(\w)}(1)$ of $\mcP(\w)$ (see \cite[\S 3.3]{ESZB23} for this and much more about heights on stacks). 
\end{remark}

Recall that the (compactified) moduli stack of elliptic curves, $\mcX_{\GL_2(\Z)}$, is isomorphic (over $K$) to the weighted projective stack $\mcP(4,6)$. This isomorphism can be given explicitly as
\begin{align*}
\mcX_{\GL_2(\Z)} &\xrightarrow{\sim} \mcP(4,6)\\
y^2=x^3+Ax+B & \mapsto [A:B].
\end{align*}
Under this isomorphism, the hyperplane bundle on $\mcP(4,6)$ corresponds to the Hodge bundle on $\mcX_{\GL_2(\Z)}$. If an elliptic curve $E$ over $K$ admits a reduced, integral, short Weierstrass model $E:y^2=x^3+Ax+B$, then the \textbf{naive height} of $E$ is defined as
\[
\Ht(E)\defeq \prod_{v\in \Val_\infty(K)} \max\{|A|_v^3,|B|_v^2\},
\]
which is the same as the height with respect to the twelfth power of the Hodge bundle. For \textit{any} short Weierstrass equation $y^2=x^3+ax+b$ of $E$, which need not be integral or reduced, the naive height is given by
\[
\Ht(E)\defeq \Ht_{(4,6)}([a:b])^{12}.
\] 
This is well defined on isomorphism classes of elliptic curves since the moduli stack of elliptic curves, $\mcX_{\GL_2(\Z)}$, is isomorphic to the weighted projective stack $\mcP(4,6)$.
The advantage of defining the naive height in this way, is that it does not require one to choose a reduced integral short Weierstrass model. 



\subsection{Defects}\label{subsec:defects}

\begin{definition}[Defect]\label{def:Defect}
Let $\varphi_f:\mcP(\w')\to\mcP(\w)$ be a morphism of weighted projective stacks with $f=(f_i)\in \Phi_{\w',\w}(K)$ where each $f_i$ has coefficients in $\O_K$, and let $[a]\in \mcP(\w')(K)$. We define the \textbf{defect} of $[a]$ with respect to $f$ to be the ideal
\[
\delta_f(a)\defeq  \frac{\mfI_\w(f(a))}{\mfI_{\w'}(a)^{e(f)}}\subseteq \O_K.
\]
\end{definition}

Defects are well defined, since for any constant $\lambda$
\[
\delta_f(\lambda\ast_{\w'} a)=\frac{\mfI_\w(f(\lambda\ast_{\w'} a))}{\mfI_{\w'}(\lambda\ast_{\w'} a)^{e(f)}}=\frac{\lambda^{e(f)}\mfI_\w(f( a))}{\lambda^{e(f)}\mfI_{\w'}( a)^{e(f)}}=\delta_f(a).
\]
From this we can show that the defect is truly an integral ideal. As each $f_i$ has integral coefficients, if we choose a representative $(a_0,\dots,a_n)$ of $[a]$ with $a_i\in \O_K$ for each $i$, then $\mfI_\w(f( a))$ will be an integral ideal. For any prime ideal $\p\subset \O_K$ we may choose our representative $(a_0,\dots,a_n)$ of $[a]$ so that it is minimal at $\p$ (i.e., such that $\min_i \lfloor v_\p(a_i)/w'_i\rfloor=0$), and hence $\p\nmid \mfI_{\w'}(a)^{e(f)}$. By the well-definedness of $\delta_f(a)$, it follows that $\p$ does not divide the denominator of the a priori fractional ideal $\delta_f(a)$. Doing this for each prime ideal of $\O_K$ shows that $\delta_f(a)$ is indeed an integral ideal.

\begin{remark}\label{rem:defects}
The defect depends on the choice of $f$. In particular, it is possible for $f,g$ to both give rise to the same morphism $\varphi:\mcP(\w')\to\mcP(\w)$, but for their sets of defects to be different. 
\end{remark}

Let $\mfD_f$ denote the \textbf{set of defects of $f$},
\[
\mfD_f\defeq \{\delta_f(a) : [a]\in \mcP(\w')(K)\}.
\]

\begin{definition}\label{def:finite-defect}
We will say that a map $f$ has \textbf{finite defect} if the set of defects $\mfD_f$ is finite. We will say that a morphism $\varphi:\mcP(\w')\to\mcP(\w)$ has \textbf{finite defect} if there exists an $f$ with finite defect such that $\varphi=\varphi_f$.
\end{definition}

Roughly speaking, set of defects helps keep track of ideals $\d$ of $\O_K$ for which there exists a $\d$-minimal representative $(a_0,\dots,a_n)$ of a point $[a]\in \mcP(\w')(K)$, with the property that the representative $f(a_0,\dots,a_n)$ of the point $[f(a)]\in \mcP(\w)(K)$ is not $\d$-minimal.

\begin{remark}
For some morphisms of weighted projective lines the set of defects was implicitly studied in the work of Cullinan, Kenney, and Voight \cite{CKV22}, and later by Cho and Jeong \cite{CJ23b}. In these cases the defects can be determined using resultants. The study of defects for more general weighted projective stacks was introduced by Manterola Ayala \cite{MA21}. 
\end{remark}

 The following example shows that not all morphisms have finite defect.

\begin{example}[Morphisms with an infinite defect set.]
\sloppy Let $w>1$, let $\w'=\w=(w,\dots, w)\in \Z_{>1}^{n+1}$, and let $e\in \Z_{>1}$. Consider the morphism
\begin{align*}
\varphi_f:\mcP(\w')(K)\to \mcP(\w)(K),
\end{align*}
induced by the map
\begin{align*}
f:K[x_0,\dots,x_n] &\to K[y_0,\dots,y_n]\\
x_i &\mapsto {y}_i^e.
\end{align*}
Then, for any principal prime ideal $(p)\subseteq \O_K$, with $p\in \O_K$, we have that
\[
\delta_f(p^{w-1})=\frac{\mfI_\w\left(f\left(p^{w-1}\right)\right)}{\mfI_{\w'}(p^{w-1})^{e(f)}}= (p)^{\left\lfloor \frac{e(w-1)}{w}\right\rfloor}/(1)=(p)^{\left\lfloor \frac{e(w-1)}{w}\right\rfloor}
\]
is a non-trivial ideal, since $e(w-1)\geq w$ as $e\geq 2$ and $w>1$. As there are infinitely many principal prime ideals in $\O_K$ (by Dirichlet's theorem for prime ideals in number fields) this shows that the set of defects $\mfD_f$ is infinite in these cases.
\end{example}

The next result gives a large class of morphisms whose set of defects is finite.

\begin{proposition}\label{prop:finite-defect}
Let $\varphi_f:\mcP(\w')\to \mcP(\w)$ be a morphism of weighted projective stacks with $e(f)=1$ or $\w'=(1,\dots,1)$. Then $f$ has finite defect.
\end{proposition}

To prove this proposition we will first prove a lemma. Define the set of prime ideals 
\[
\widetilde{\mfD}_f\defeq \{\p\subset \O_K : \p|\d \text{ for some } \d\in \mfD_f\}.
\]

\begin{lemma}\label{lem:defect_primes}
If $\w'=(1,\dots,1)$ or $e(f)=1$, then the set $\widetilde{\mfD}_f$ is finite.
\end{lemma}

\begin{proof}
Without loss of generality, we may order the $(n+1)$-tuples $\w'=(w'_0,\dots,w'_n)$ and $\w=(w_0,\dots,w_n)$ so that $w'_i\leq w'_j$ and $w_i\leq w_j$ for all $i\leq j$. Let $i_0=0, i_1,\dots, i_m$ be the set of indices for which
\[
w'_0=\cdots=w'_{i_1-1}\ <\ w'_{i_1}=\cdots=w'_{i_2-1}\ <\ \cdots\ <\ w'_{i_m}=\cdots=w'_n.
\] 
Let $\varphi_f:\mcP(\w')\to\mcP(\w)$ be the morphism of weighted projective stacks corresponding to the $(n+1)$-tuple of weighted homogeneous polynomials $f=(f_0,\dots,f_n)$. By possibly rescaling, we may assume that the $f_i$ have integral coefficients. For any $(n+1)$-tuple of non-negative integers $z=(z_0,\dots,z_n)$, let $x^z\defeq x_0^{z_0}x_1^{z_1}\cdots x_n^{z_n}$. Then there exist algebraic integers $\alpha_{i,z}\in \O_K$ for which
\[
f_i(x)=\sum_{\substack{ z=(z_0,\dots,z_n)\in \Z^{n+1}_{\geq 0}\\ z_0w'_0+\cdots+z_n w'_n=e(f) w_i}} \alpha_{i,z} x^z.
\] 

As $\sqrt{(f_0,\dots,f_n)}=(x_0,\dots,x_n)$, there exist positive integers $r_i\in \Z_{>0}$, non-zero algebraic integers $\alpha_i\in \O_K-\{0\}$, and polynomials $g_{i,j}\in \O_K[x_1,\dots,x_n]$, such that, for each $i\in \{0,\dots, n\}$,
\begin{align}\label{eq:radical}
\alpha_{i} x_i^{r_i}=\sum_{j=0}^n g_{i,j}(x) f_j(x). 
\end{align}

Let $\mfA\subseteq \O_K$ be the integral ideal  obtained by taking the intersection of the principal ideals $(\alpha_{i})$ and $(\alpha_{j,z})$. To prove the lemma it suffices to prove that each prime ideal in the set $\widetilde{\mfD}_f$ divides $\mfA$, and to do this it suffices to prove that any prime ideal which does not divide $\mfA$ is not in $\widetilde{\mfD}_f$. So let $\p$ be a prime ideal of $\O_K$ that does not divide the ideal $\mfA$. Let $a\in \mcP(\w')(K)$ and let $(a_0,\dots,a_n)$ be a representative of $a$ which is minimal at $\p$ (i.e., such that $\min_i \lfloor v_\p(a_i)/w'_i\rfloor=0$ for all $i$). 
 For the remainder of the proof we will work in the completion $K_\p$.

To prove the lemma it will suffice to show that, if $v_\p(f_i(a))\geq w_i$ for all $i$, then $v_\p(a_i)\geq {w'_i}$ for all $i$. Assuming that $v_\p(f_i(a))\geq w_i$ for all $i$, equation (\ref{eq:radical}) implies that $v_\p(a_i)\geq 1$ for all $i$. Thus, in the case that $\w'=(1,1\dots,1)$ we are done. So suppose that not all weights $w'_i$ are equal to one, and assume that $e(f)=1$. Let $a_{i,1}$ be such that $a_i=\pi_\p a_{i,1}$.

First we consider the $n-i_m$ weights $w'_{i_m},w'_{i_m+1},\dots, w'_n$, which are all equal to $w'_n$. Define
\[
f_j^{(m)}(x_{i_m}, x_{i_m+1},\dots,x_n)\defeq f_j(0,\dots,0,x_{i_m},x_{i_m+1},\dots,x_n)
\]
and
\[
g_{i,j}^{(m)}(x_{i_m}, x_{i_m+1},\dots,x_n)\defeq g_{i,j}(0,\dots,0,x_{i_m},x_{i_m+1},\dots,x_n).
\]
By equation (\ref{eq:radical}), for $i\geq i_m$, we have 
\begin{align}\label{eq:radical(m)}
\alpha_i x_i^{r_i} = \sum_{j=0}^n g_{i,j}^{(m)}(x) f_j^{(m)}(x).
\end{align}
As $w'_{i_m},\dots, w'_n$ are the largest weights, the terms of the form $\alpha_{j,z} x_{i_m}^{z_{i_m}}\cdots x_{n}^{z_{n}}$ are those of lowest unweighted degree in the weighted homogeneous polynomials $f_j(x)$. In particular, these terms have unweighted degree $w_j/w'_n$. From this, since $v_\p(a_i)\geq 1$ for all $i$ and $v_\p(f_j(a))\geq w_j$ for all $j$, it follows that
\begin{align}\label{eq:divisibility_m}
v_\p(f_j^{(m)}(a_{i_m,1}\dots,a_{n,1}))\geq w_j-w_j/w'_n
\end{align}
for all $j$. By equation (\ref{eq:radical(m)}), this implies $v_\p(a_{i,1})\geq 1$ for all $i\geq i_m$, and thus $v_\p(a_i)\geq 2$. If $w'_n=2$ we are done. So suppose $w'_n>2$. For $i\geq i_m$ let $a_{i,2}$ be such that $a_i=\pi_\p^2 a_{i,2}$. Note that the non-zero $f_j^{(m)}$ are homogeneous of (non-weighted) degree $w_j/w_n'$. Therefore, by equation (\ref{eq:divisibility_m}) and the assumption $w'_n>2$, we have
\[
v_\p(f_j^{(m)}(a_{i_m,2},\dots,a_{n,2}))\geq w_j-2w_j/w'_n >0.
\]
By equation (\ref{eq:radical(m)}), this implies  $v_\p(a_{i,2})\geq 2$ for all $i\geq i_m$, and thus $v_\p(a_i)\geq 3$. If $w'_n=3$ we are done. Continuing in this way, we will eventually have that $v_\p(a_i)\geq w'_i$ for all $i\geq i_m$.

We now consider the weights $w'_{i_{m-1}},\dots, w'_{i_{m}-1}$, which are all equal to $w'_{i_{m-1}}$. Let 
\[
f_j^{(m-1)}(x_{i_{m-1}}, x_{i_{m-1}+1},\dots,x_{i_{m}-1})\defeq f_j(0,\dots,0,x_{i_{m-1}}, x_{i_{m-1}+1},\dots,x_{i_{m}-1},0,\dots,0)
\]
and
\[
g_{i,j}^{(m-1)}(x_{i_{m-1}}, x_{i_{m-1}+1},\dots,x_{i_{m}-1})\defeq g_{i,j}(0,\dots,0,x_{i_{m-1}}, x_{i_{m-1}+1},\dots,x_{i_{m}-1},0,\dots,0).
\]
Then, by equation (\ref{eq:radical}), for $i\geq i_m$ we have 
\begin{align}\label{eq:radical(m-1)}
\alpha_i x_i^{r_i} = \sum_{j=0}^n g_{i,j}^{(m-1)}(x) f_j^{(m-1)}(x).
\end{align}
Observe that all coefficients of the polynomials
\[
f_j(\pi_\p x_{0},\dots, \pi_\p x_{i_{m-1}}, \pi_\p^{w'_n} x_{i_m},\dots,\pi_\p^{w'_n} x_{n})\in K[x_{0}, x_{1},\dots, x_{n}]
\]
are divisible by $\pi_\p^{w_j/w'_{i_m}}$, and the only coefficients divisible by $\pi_\p^{w_j/w'_{i_m}}$ exactly are those of the terms $x_{i_{m-1}}^{z_{i_{m-1}}}\cdots x_{i_{m-1}-1}^{z_{i_m-1}}$. From this, together with the fact that $v_\p(f_j(a))\geq w_j$ for all $j$, it follows that
\begin{align}\label{eq:divisibility_m-1}
v_\p( f_{j}^{(m-1)}(a_{i_{m-1},1},\dots, a_{i_{m}-1}))\geq w_j-w_j/w'_{i_{m-1}}
\end{align}
for all $j$. By equation (\ref{eq:radical(m-1)}), this implies that $v_\p(a_{i,1})\geq 1$ for all $i_{m-1}\leq i < i_m$, and thus $v_\p(a_{i,1})\geq  2$. So if $w'_{i_{m-1}}=2$, then we are done. Therefore, suppose $w'_{i_{m-1}}>2$.  For $i_{m-1}\leq i < i_m$ let $a_{i,2}$ be such that $a_i=\pi_\p^2 a_{i,2}$. As the $f_j^{(m-1)}$ are homogeneous of (non-weighted) degree $w_j/w'_{i_{m-1}}$, it follows from equation (\ref{eq:divisibility_m-1}) and the assumption $w'_{i_{m-1}}>2$, that
\[
v_\p(f_j^{(m-1)}(a_{i_{m-1},2},\dots, a_{i_m-1,2}))\geq w_j-2 w_j/w'_{i_{m-1}}>0.
\]
By equation (\ref{eq:radical(m-1)}), this implies that $v_\p(a_{i,2})\geq 1$ for all $i_{m-1}\leq i < i_m$, and therefore $v_\p(a_i)\geq 3$. Therefore we are done if $w'_{i_{m-1}}=3$. Continuing in this way, we eventually find that $v_\p(a_i)\geq w'_i$ for all $i_{m-1}\leq i < i_m$. 

Repeating the above argument with $i_{m-1}$ replaced by each of $i_{m-2},\dots, i_0$ shows that $v_\p(a_i)\geq w'_i$ for each $i$, which is what we wanted to show.
\end{proof}

\begin{proof}[Proof of Proposition \ref{prop:finite-defect}]

By the proof of Lemma \ref{lem:defect_primes} it follows that each $\d\in \mfD_f$ is divisible only by the finitely many primes which divide the integral ideal $\mfA\subseteq \O_K$, generated by all the $\alpha_i$ and $\alpha_{j,z}$. 

Let $\p$ be a prime ideal dividing $\mfA$. To prove the proposition it remains to show that the set of positive integers
\begin{align}\label{eq: s-finite}
\{s: \p^s|\d \text{ for some } \d\in \mfD_f\}
\end{align}
is bounded. By equation (\ref{eq:radical}), for any $s\geq \max_i \{r_i w'_i+v_\p(\alpha_i)\}$ we will have that, if $v_\p(f_j(a))\geq s w_j$ for all $j$, then $v_\p(a_i)\geq w'_i$ for all $i$. It follows that the set (\ref{eq: s-finite}) is finite, and $\mfD_f$ is contained within the finite set of ideals 
\[
\left\{\prod_{\p|\mfA} \p^{s_\p} : 0\leq s_\p<\max_i \{r_i w'_i+v_\p(\alpha_i)\}\right\}.
\]
\end{proof}

We complete this section with a proposition comparing the weights $\w$ and $\w'$.

\begin{proposition}\label{prop:w'<e(f)w}
Let $\varphi_f : \mcP(\w')\to\mcP(\w)$ be a morphism of weighted projective stacks, with weighted degree $e(f)$. Suppose that the weights $\w$ and $\w'$ are  ordered so that $w_i\leq w_j$ and $w'_i\leq w'_j$ for all $i\leq j$. Then $w_i'\leq e(f) w_i$ for all $i$.
\end{proposition}

\begin{proof}
By way of contradiction, suppose that $w'_s> e(f)w_s$ for some $s$, and assume that $s$ is minimal amongst the indices with this property. Let
\[
f_j^{(s)}(x_s, x_{s+1},\dots, x_{n})\defeq f_j(0,\dots, 0, x_{s},x_{s+1},\dots, x_n) 
\]
and
\[
g_{i,j}^{(s)}(x_{s}, x_{s+1},\dots, x_n)\defeq g_{i,j}(0,\dots,0,x_{s}, x_{s+1},\dots, x_n).
\] 
Note that $w'_i>e(f)w_s$ for all $i\geq s$, so that $f_j^{(s)}=0$ for all $j\leq s$.
Therefore, by equation (\ref{eq:radical}), we have that
\begin{align}\label{eq:radical(s)}
\alpha_i x_i^{r_i} = \sum_{j=s+1}^n g_{i,j}^{(s)}(x) f_j^{(s)}(x).
\end{align}
 But this is impossible since there are at most $n-s$ non-zero $f_j^{(s)}$, but there are $n-s+1$ non-zero $x_i$.
\end{proof}

\section{Weighted geometric sieve over number fields}\label{sec:Geometric-Sieve}

In this section we prove a weighted version of some results of Bright, Browning, and Loughran \cite[\S 3]{BBL16}, which are proven along the same lines as a lemma of Poonen and Stoll \cite[Lemma 20]{PS99}. 

Throughout this section $\w= (w_1,w_2,\dots, w_n)$ will be an $n$-tuple of positive integers, $|\w|\defeq w_1+w_2+\cdots +w_n$, and $w_{\min}=\min_i\{w_i\}$. Similarly, $\w'$ will be an $n$-tuple of positive integers with $|\w'|$ and $w'_{\min}$ defined analogously. 

\subsection{O-minimal geometry}\label{subsec:GSo-minimal}

We briefly cover the basics of o-minimal geometry which will be needed in this article. A general reference for this subsection is \cite[\S 1]{vdD98}.
Let $m_L$ denote Lebesgue measure on $\R^n$.

\begin{definition}[Semi-algebraic set]
A (real) \textbf{semi-algebraic subset} of $\R^n$ is a finite union of sets of the form
\[
\{x\in \R^n: f_1(x)=\cdots=f_k(x)=0 \text{ and } g_1(x)>0,\dots,g_l(x)>0\},
\]
where $f_1,\dots,f_k,g_1,\dots,g_l\in \R[X_1,\dots,X_n]$.
\end{definition}

\begin{definition}[Structure]
A \textbf{structure} is a sequence, $\mcS=(\mcS_n)_{n\in \Z_{>0}}$, where each $\mcS_n$ is a set of subsets of $\R^n$ with the following properties:
\begin{enumerate}[label=(\roman*)]
\item If $A,B\in \mcS_n$, then $A\cup B\in \mcS_n$ and $\R^n - A\in \mcS_n$ (i.e., $\mcS_n$ is a Boolean algebra).
\item If $A\in \mcS_m$ and $B\in \mcS_n$, then $A\times B\in \mcS_{n+m}$.
\item If $\pi:\R^n\to\R^m$ is the projection to $m$ distinct coordinates and $A\in \mcS_n$, then $\pi(A)\in \mcS_m$.
\item All real semi-algebraic subsets of $\R^n$ are in $\mcS_n$.
\end{enumerate}
A subset is \textbf{definable} in $\mcS$ if it is contained in some $\mcS_n$. Let $D\subseteq S_n$. A function $f\colon D\to \R^m$ is said to be \textbf{definable} in $\mcS$ if its graph, $\Gamma(f)\defeq\{(x,f(x)):x\in D\}\subseteq \R^{m+n}$, is definable in $\mcS$.
\end{definition}

Note that the intersection of definable sets is definable by property (i).

\begin{definition}[O-minimal structure]
An \textbf{o-minimal structure} is a structure in which the following additional property holds:
\begin{enumerate}[label=(v)]
\item The boundary of each set in $\mcS_1$ is a finite set of points.
\end{enumerate}
\end{definition}

The class of semialgebraic sets is an example of an o-minimal structure (see, e.g., \cite[\S 2]{vdD98}).
The main structure we will use is $\R_{\exp}$, which is defined to be the smallest structure in which the real exponential function, $\exp\colon\R\to\R$, is definable. Observe that the function $\log(x)$ is definable in $\R_{\exp}$ since its graph,
\[
\Gamma(\log(x))=\{(x,\log(x)) : x\in \R_{>0}\}=\{(\exp(y), y) : y\in \R\}\subset \R^2,
\]
is clearly definable in $\R_{\exp}$.

\begin{theorem}[\cite{Wil96}]\label{thm:RexpO-minimal}
The structure $\R_{\exp}$ is o-minimal.
\end{theorem}

From now on we will call a subset of $\R^n$ \textbf{definable} if it is definable in some o-minimal structure.

 The following proposition gives the key measure theoretic properties of bounded definable subsets in $\R^n$.

\begin{proposition}{\cite[Lemma 5.3]{BW14}}\label{prop:DefinableMeasurable}
If $R\subset \R^n$ is a bounded definable subset, then $R$ is measurable and the measure of its boundary $\partial R$ is $m_L(\partial R)=0$.
\end{proposition}

\subsection{Weighted geometry-of-numbers}\label{subsec:GSgeometry-of-numbers}

The following is a version of the Principle of Lipschitz, which gives estimates for the number of lattice points in a weighted homogeneous space:

\begin{proposition}\label{prop:o-Lip}
Let $R\subset \R^{n}$ be a bounded definable set. Let  $\Lambda\subset \R^n$ be a rank $n$ lattice. For any positive real number $B$, define the set 
\[
R(B)\defeq B\ast_{\w'}R=\{B\ast_{\w'}x : x\in R\}.
\] 
Then
\[
\#(\Lambda\cap R(B))=\frac{m_L(R)}{\det \Lambda} B^{|\w'|} + O\left(B^{|\w'|-w'_{\min}}\right),
\]
as a function of $B$, where the implied constant depends only on $R$, $\Lambda$, and $\w'$. 
\end{proposition}

\begin{proof}
This follows from a general version of the Principle of Lipschitz due to Barroero and Widmer \cite[Theorem 1.3]{BW14}. Since $R(B)$ is the image of $B$ under the linear transformation
\begin{equation*}
	M=\begin{pmatrix}
	 B^{w'_{1}} & & \\ & \ddots & \\ & & B^{w'_{n}}			   \end{pmatrix},
\end{equation*}
by \cite[Theorem 2.20 (e) and \S 2.23]{Rud87} we have
\begin{equation*}
 	m_L(R(B))=\det(M) m_L(R) = B^{|\w'|} m_L(R).
 \end{equation*}
Therefore, \cite[Theorem 1.3]{BW14} gives the desired leading term. Let $V_j(R(B))$ denote the sum of the $j$-dimensional volumes of the orthogonal projections of $R(B)$ onto each $j$-dimensional coordinate subspace of $\R^n$. The error term given in \cite[Theorem 1.3]{BW14} is
\begin{align}\label{eq:o-lip error}
O\left(1+\sum_{j=1}^{n-1} V_j(R(B))\right),
\end{align}
where the implied constant depends only on $R$ and $\Lambda$.
 In our case, one observes that $V_i(R(B))=O(V_{j}(R(B))$ for all $i\leq j$. Moreover, 
 \[
 V_{n-1}(R(B))=O(\sum_{i\leq n} B^{|\w'|-w_i'})=O(B^{|\w'|-w'_{\min}}),
 \]
  where the implied constant depends only on $R$ and ${\w'}$. By dimension considerations, we see that if $B^{|\w'|-w'_{\min}}=O(V_i(R(B)))$, then we must have $i\geq n-1$. Combining these observations with the error term (\ref{eq:o-lip error}) we obtain the desired error term.
\end{proof}

We define a \textbf{($\Z_p^n$)-box} $\mcB_p\subset \Z_p^n$ to be a Cartesian product $\prod_{j=1}^n \mcB_{p,j}$ of finitely many closed balls $\mcB_{p,j}\defeq \{x\in \Z_p : |x-a_{p,j}|_p\leq b_{p,j}\}\subset \Z_p$, where $a_{p,j}\in \Z_p$ and $b_{p,j}\in \{p^k: k\in \Z\}$. Let $S$ be a finite set of primes. We define a \textbf{box} $\mcB\subset \prod_{p\in S} \Z_p^n$ to be a Cartesian products of finitely many $\Z_p^n$-boxes,
\[
\mcB=\prod_{p\in S} \mcB_{p}=\prod_{p\in S}\prod_{j=1}^n \mcB_{p,j}.
\]


\begin{lemma}[Box Lemma]\label{lem:box}
Let $\Omega_\infty\subset \R^n$ be a bounded definable subset. Let $S$ be a finite set of primes and $\mcB=\prod_{p\in S} \mcB_p$ a box (as defined above). Then, for each positive real number $B\in \R_{>0}$, we have
\begin{align*}
\#\left\{x\in \Z^n \cap B\ast_\w \Omega_\infty: x\in \prod_{p\in S} \mathcal{B}_p \right\}
=\left(m_L(\Omega_\infty)\prod_{p\in S} m_p(\mathcal{B}_p)\right)B^{|\w|}
 + O\left(B^{|\w|-w_{\min}} \right),
\end{align*}
as a function of $B$, where the implied constant depends only on $\Omega_\infty$, $\mathcal{B}$, and $\w$.
\end{lemma}

\begin{proof}
Note that the intersection $\Z^n\cap \prod_{p\in S} \mathcal{B}_p$ is a translate of a sub-lattice of $\Z^n$. This lattice has determinant $\prod_{p\in S}m_p(\mathcal{B}_p)^{-1}$.
 The lemma then follows from Proposition \ref{prop:o-Lip}.
\end{proof}


We now give a version of Lemma \ref{lem:box} which applies for more general sets $\Omega_\infty$, but without an error term.

\begin{corollary}\label{cor:box}
Let $\Omega_\infty\subset \R^n$ be a bounded measurable subset with $m_L(\partial \Omega_\infty)=0$. Let $S$ be a finite set of primes and $\mcB=\prod_{p\in S} \mcB_p$ a box. Then we have that
\begin{align*}
\#\left\{x\in \Z^n \cap B\ast_\w \Omega_\infty: x\in \prod_{p\in S} \mathcal{B}_p \right\}\sim\left(m_L(\Omega_\infty)\prod_{p\in S} m_p(\mathcal{B}_p)\right)B^{|\w|},
\end{align*}
as a function of $B$.
\end{corollary}

\begin{proof}
Since the closure $\overline{\Omega}_\infty$ of $\Omega_\infty$ is compact, it can be covered by a finite collection of definable sets $(R_i)_{i\in I}$ (e.g., by boxes in $\R^n$) such that the measure of the union of those definable sets $m_\infty\left(\bigcup R_i\right)$ is arbitrarily close to $m_\infty(\Omega_\infty)$. Applying Lemma \ref{lem:box} to each $R_i$ and summing over the indices $i\in I$ gives the desired asymptotic.
\end{proof}

Let $\Lambda$ be a free $\Z$-module of finite rank and let $\Lambda_\infty\defeq \Lambda\otimes_\Z \R$ and $\Lambda_p\defeq \Lambda\otimes_\Z \Z_p$. Equip $\Lambda_\infty$ and $\Lambda_p$ with Haar measures $m_\infty$ and $m_p$, normalized so that $m_p(\Lambda_p)=1$ for all but finitely many primes $p$.

 An \textbf{(affine) local condition} at a prime $p$ will refer to a subset $\Omega_p\subseteq \Lambda_p$. We are now going to prove two lemmas for counting lattice points in weighted homogeneous spaces with prescribed local conditions. The two cases will be \textit{finitely many local conditions} and \textit{infinitely many local conditions}. The case of infinitely many local conditions addresses a more general situation, but at the cost of not giving an error term. Later in this section we will definite local conditions over number fields.

\begin{lemma}\label{lem:ZnCRTfin}
Let $\Lambda$ be a free $\Z$-module of rank $n$, let $\Omega_\infty\subset \Lambda_\infty$ be a bounded definable subset, and let $S$ be a finite set of primes. For each prime $p$ in the set $S$, let $\Omega_p\subset \Lambda_p$ be a measurable subset with $m_p(\partial\Omega_p)=0$.
Then 
\begin{align*}
&\#\{x\in \Lambda\cap (B \ast_\w \Omega_\infty): x\in \Omega_p \text{ for all primes } p\in S\}\\
&\hspace{5mm} =
\left(\frac{m_\infty(\Omega_\infty)}{m_\infty(\Lambda_\infty/\Lambda)}\prod_{p\in S} \frac{m_p(\Omega_p)}{m_p(\Lambda_p)}\right)B^{|\w|} + O\left(\frac{B^{|\w|-w_{\min}}}{m_\infty(\Lambda_\infty/\Lambda)} \right),
\end{align*}
as a function of $B$, where the implied constant depends only on $\Lambda$, $\Omega_\infty$, $\Omega_p$, and $\w$.
\end{lemma}

\begin{proof}
Fix an isomorphism $\Lambda\cong \Z^n$. The measures $m_\infty$ and $m_p$ on $\Lambda_\infty$ and $\Lambda_p$ induce measures on $\R^n$ and $\Z_p^n$ which differ from the usual Haar measures by $m_\infty(\Lambda_\infty/\Lambda)$ and $m_p(\Lambda_p)$ respectively. It therefore suffices to prove the result in the case $\Lambda=\Z^n$, $m_\infty=m_L$, and $m_p$ are the usual Haar measures.

 Let $P$ denote the product $\prod_{p\in S} \Omega_p$ and let $Q=\bigcup_{p\in S} \left(\Z_p^n-\Omega_p\right)\prod_{q\in S-\{p\}}\Z_q^n$ be the complement of $P$ in $\Z_p^n$. 
 By assumption, the boundary of $\Omega_p$, $\partial \Omega_p$, which coincides with the boundary of its complement, $\partial(\Z_p^n-\Omega_p)$, has measure zero for all $p\in S$. Thus, by compactness, we may cover the closure of $P$ by a finite number of boxes, $\left(\mcI^{(i)}\right)_{i\in I}$, such that the sum of their measures is arbitrarily close to $\prod_{p\in S} m_p(\Omega_p)$.
Similarly, we may cover the closure of $Q$ by finitely many boxes, $\left(\mcJ^{(j)}\right)_{j\in J}$, such that the sum of their measures is arbitrarily close to $1-\prod_{p\in S} m_p(\Omega_p)$.

 Applying the Box Lemma (Lemma \ref{lem:box}) to the boxes $(\mcI^{(i)})_{i\in I}$, and then summing over the boxes, gives an upper bound for the set
\[
\#\{x\in \Lambda\cap B \ast_\w \Omega_\infty: x\in \Omega_p \text{ for all primes } p\in S\}.
\]
Similarly, applying the Box Lemma to the boxes $(\mcJ^{(j)})_{j\in J}$ and summing, gives an upper bound for the complement of the above set. Together these bounds imply the desired asymptotic.
\end{proof}

Replacing the use of Lemma \ref{lem:box} with Corollary \ref{cor:box} in the above proof, we obtain the following corollary:

\begin{corollary}\label{cor:ZnCRTfin}
Let $\Lambda$ be a free $\Z$-module of finite rank $n$, let $\Omega_\infty\subset \Lambda_\infty$ be a bounded measurable subset with $m_\infty(\partial \Omega_\infty)=0$, and, for each prime $p$ in a finite subset $S$, let $\Omega_p\subset \Lambda_p$ be a measurable subset with $m_p(\partial\Omega_p)=0$.
Then 
\[
\#\{x\in \Lambda\cap B \ast_\w \Omega_\infty: x\in \Omega_p \text{ for all primes } p\in S\}\sim
\left(\frac{m_\infty(\Omega_\infty)}{m_\infty(\Lambda_\infty/\Lambda)}\prod_{p\in S} \frac{m_p(\Omega_p)}{m_p(\Lambda_p)}\right)B^{|\w|}.
\]
\end{corollary}

We now address the case of infinitely many local conditions.

\begin{lemma}\label{lem:ZnCRTinfty}
 Let $\Lambda$ be a free $\Z$-module of finite rank $n$, let $\Omega_\infty\subset \Lambda_\infty$ be a bounded measurable subset, and for each prime $p$ let $\Omega_p\subset \Lambda_p$ be a measurable subset. Assume that the measures of the boundaries $m_\infty(\partial \Omega_\infty)$ and $m_p(\partial\Omega_p)$ equal zero,
and that
\begin{align}\label{eq:ZnEk}
\lim_{M\to\infty}\limsup_{B\to\infty} \frac{\#\{x\in \Lambda\cap B\ast_\w \Omega_\infty : x\not\in \Omega_p \text{ for some prime }p>M\}}{B^{|\w|}}=0. 
\end{align}
Then 
\[
\#\{x\in \Lambda\cap B \ast_\w \Omega_\infty: x\in \Omega_p \text{ for all primes } p\}\sim\frac{m_\infty(\Omega_\infty)}{m_\infty(\Lambda_\infty/\Lambda)} \left(\prod_p \frac{m_p(\Omega_p)}{m_p(\Lambda_p)}\right)B^{|\w|},
\]
as a function of $B$.
\end{lemma}

\begin{proof}
As in the proof of Lemma \ref{lem:ZnCRTfin}, we may reduce to the case in which $\Lambda=\Z^n$ and $m_\infty$ and $m_p$ are the usual Haar measures on $\R^n$ and $\Z_p^n$.

 For $M\leq M'\leq \infty$ and $B>0$, define the function
\[
f_{M,M'}(B)\defeq \frac{1}{B^{|\w|}}\#\{x\in \Z^n\cap B\ast_\w\Omega_\infty : x\in \Omega_p \text{ for all primes } p\in [M,M')\}
\]
and let $f_M(B)\defeq f_{1,M}(B)$. Note that $f_M(B)\geq f_{M+r}(B)$ for all $r\in\Z_{\geq 0} \cup \{\infty\}$. The hypothesis (\ref{eq:ZnEk}) implies that
\begin{align}\label{eq:pfZnEk1}
\lim_{M\to\infty} \limsup_{B\to\infty} (f_M(B)-f_\infty(B))=0.
\end{align}
From Corollary \ref{cor:ZnCRTfin} it follows that, for all $M<M'<\infty$,
\begin{align}\label{eq:pfZnEk2}
\lim_{B\to \infty} f_{M,M'}(B)=m_\infty(\Omega_\infty) \prod_{M\leq p<M'} m_p(\Omega_p).
\end{align}
As $f_\infty(B)$ is non-decreasing, 
\[
\limsup_{B\to \infty} f_\infty(B) = \lim_{B\to \infty} f_\infty(B).
\]
Combining this with the limits (\ref{eq:pfZnEk1}) and (\ref{eq:pfZnEk2}) gives
\[
\lim_{B\to\infty} f_\infty(B)=\lim_{M\to\infty}\lim_{B\to\infty} f_M(B)=m_\infty(\Omega_\infty)\lim_{M\to\infty}\prod_{p<M} m_p(\Omega_p).
\]

The above infinite product converges by Cauchy's criterion since the limits (\ref{eq:ZnEk}) and (\ref{eq:pfZnEk2}) together imply that
\[
\lim_{M\to\infty}\sup_{r\in \Z_{\geq 1}}\left|1-\prod_{M\leq p<M+r} m_p(\Omega_p)\right|=\frac{1}{m_\infty(\Omega_\infty)} \lim_{M\to\infty} \sup_{r\in \Z_{\geq 1}}\lim_{B\to\infty} |f_M(B)-f_{M,M+r}(B)|=0.
\]
The desired result then follows.
\end{proof}

If $K$ is a number field of degree $d$ over $\Q$ with $r_1$ real embeddings and $r_2$ pairs of complex embeddings into $\C$ and with discriminant $\Delta_K$, then its ring of integers $\O_K$ may naturally be viewed as a rank $d$ lattice in $K_\infty\defeq \O_K\otimes_{\Z} \R=\prod_{v|\infty} K_v$, with covolume $|\Delta_K|^{1/2}$ with respect to the usual Haar measure $m_\infty$ on $K_\infty$ (which differs from Lebesgue measure on $K_\infty\cong \R^{r_1+2r_2}$ by a factor of $2^{r_2}$) \cite[Chapter I Proposition 5.2]{Neu99}. 
More generally, any integral ideal $\a\subseteq \O_K$ may be viewed as a lattice in $K_\infty$ with covolume $N_{K/\Q}(\a)|\Delta_K|^{1/2}$. Define the lattice
\[
\a^\w\defeq \a^{w_1}\times\cdots \times \a^{w_n}\subset K_\infty^{n}.
\]
This lattice has covolume $N_{K/\Q}(\a)^{|\w|}|\Delta_K|^{n/2}$. For example, in the case that $\a=\O_K$ and $w_i=1$ for all $i$, one has the lattice $\Lambda_\a=\O_K^{n}$ of covolume $|\Delta_K|^{n/2}$.

In later applications we will be interested in certain sets of lattices. 

\begin{definition}
Let $\mfD\subseteq \O_K$ be an ideal. A \textbf{$\mfD$-admissible system of lattices} $(\Lambda(\a))_{\a\subseteq \O_K}$ is a set of sublattices $\Lambda(\a)\subset \a^\w$ indexed by ideals $\a\subset \O_K$, with the property that if $\a,\c,\n_0,\n_1 \subseteq \O_K$ are ideals satisfying
 \begin{itemize}
 	\item $\a=\c\n_0\n_1$,
 	\item $\n_0$ relatively prime to $\mfD$, and
 	\item $\n_1\subseteq \mfD$,  	
 \end{itemize}
  then $\Lambda(\a) = \Lambda(\c\n_1) \cap (\c\n_0)^{\w}$.
\end{definition}

 For $(\Lambda(\a))_{\a\subseteq \O_K}$ a $\mfD$-admissible systems of lattices, 
\begin{align*}
m_\infty(K_\infty^{n}/\Lambda(\a))
&=[\c^\w:\Lambda(\a)]\cdot m_\infty(K_\infty^{n}/\c^\w)\\
&=[\c^\w:\Lambda(\c\n_1)\cap (\c\n_0)^{\w}]\cdot m_\infty(K_\infty^{n}/\c^\w)\\
&=[\c^\w:(\c\n_0)^\w]\cdot [\c^\w : \Lambda(\c\n_1)] \cdot N_{K/\Q}(\c)^{|\w|}|\Delta_K|^{n/2}\\
&=N_{K/\Q}(\n_0)^{|\w|} \frac{[\O_K^{n}:\Lambda(\c\n_1)]}{N_{K/\Q}(\c)^{|\w|}} N_{K/\Q}(\c)^{|\w|}|\Delta_K|^{n/2}\\
&=N_{K/\Q}(\n_0)^{|\w|} [\O_K^{n}:\Lambda(\c\n_1)]\cdot |\Delta_K|^{n/2}.
\end{align*}

 For any rational prime $p$, we have $\O_K\otimes_{\Z} \Z_p=\prod_{\p|p} \O_{K,\p}$. Equip each $\O_{K,\p}$ with the Haar measure $m_\p$ normalized so that $m_\p(\O_{K,\p})=1$. These measures induce measures on $\O_{K,\p}^n$, which we will also denote by $m_\p$.
 
 An \textbf{(affine) local condition} at a finite place $v\in \Val_0(K)$ will refer to a subset $\Omega_v\subseteq \O_{K,\p_v}^n$. We will sometimes write $\Omega_{\p_v}$ for $\Omega_v$.
 
Applying Lemma \ref{lem:ZnCRTfin} with $\Lambda_\p=\O_{K,\p}^n$, we obtain the following:

\begin{proposition}\label{prop:KCRTfin}
Let $K/\Q$ be a number field of discriminant $\Delta_K$ and degree $d$ over $\Q$. Let $(\Lambda(\a))_{\a\subseteq \O_K}$ be a $\mfD$-admissible system of lattices, and let $\a,\c,\n_0,\n_1\subseteq \O_K$ be integral ideals such that $\Lambda(\a)=\Lambda(\c\n_1)\cap (\c\n_0)^{\w}$. Let $\Omega_\infty\subset K_\infty^{n}$ be a bounded definable subset. Let $S$ be a finite set of prime ideals of $\O_K$, and for each $\p\in S$ let $\Omega_\p\subset \O_{K,\p}^n$ be a subset whose boundary 
$\partial \Omega_\p$ has measure $m_\p(\partial \Omega_\p)=0$.
Then 
\begin{align*}
&\#\{x\in \Lambda(\a)\cap (B \ast_\w \Omega_\infty) : x\in \Omega_\p \text{ for all primes } \p\in S\}\\
&\hspace{0cm}=\frac{m_\infty(\Omega_\infty)}{N_{K/\Q}(\n_0)^{|\w|} [\O_K^{n}:\Lambda(\c\n_1)]\cdot |\Delta_K|^{n/2}} \left(\prod_{\p\in S} m_\p(\Omega_\p)\right)B^{d|\w|}+ 
 O\left(B^{d|\w|-w_{\min}} \right),
\end{align*}
as a function of $B$, where the implied constant depends only on $K$, $\Omega_\infty$, $\Omega_\p$, and $\w$.
\end{proposition}

Similarly, applying Lemma \ref{lem:ZnCRTinfty} gives the corresponding proposition for infinitely many local conditions.

\begin{proposition}\label{prop:KEk}
Let $K/\Q$ be a number field of discriminant $\Delta_K$ and degree $d$ over $\Q$. 
Let $(\Lambda(\a))_{\a\subseteq \O_K}$ be a $\mfD$-admissible system of lattices, and let $\a,\c,\n_0,\n_1\subseteq \O_K$ be integral ideals such that $\Lambda(\a)=\Lambda(\c\n_1)\cap (\c\n_0)^{\w}$.
 Let $\Omega_\infty\subset K_\infty^n$ be a bounded subset and let $\Omega_\p\subset \O_{K,\p}^n$ for each prime ideal $\p$ of $\O_K$. Assume that $m_\infty(\partial \Omega_\infty)=0$ and $m_\p(\partial \Omega_\p)=0$ for each prime $\p$
and that
\begin{align}\label{eq:KEk}
\lim_{M\to\infty}\limsup_{B\to\infty} \frac{\#\{x\in \Lambda(\a)\cap (B\ast_\w \Omega_\infty) : x\not\in \Omega_{\p} \text{ for some prime $\p$ with }N_{K/\Q}(\p)>M\}}{B^{d|\w|}}=0.
\end{align}
Then 
\begin{align*}
&\#\{x\in \Lambda(\a)\cap (B \ast_\w \Omega_\infty): x\in \Omega_\p \text{ for all primes } \p\}\\
&\qquad \sim\frac{m_\infty(\Omega_\infty)}{N_{K/\Q}(\n_0)^{|\w'|} [\O_K^n : \Lambda(\c\n_1)]\cdot |\Delta_K|^{n/2}} \left(\prod_\p m_\p(\Omega_\p)\right) B^{d|\w|},
\end{align*}
as a function of $B$.
\end{proposition}

The following lemma gives a useful criterion for checking the limit condition (\ref{eq:KEk}) of Proposition \ref{prop:KEk}.

\begin{lemma}\label{lem:Ek-Bha}
Let $K$ be a number field and $Y\subset \A_{\O_K}^n$ a closed subscheme of codimension $k>1$. Let $\Omega_\infty\subset K_\infty^n$ be a bounded subset with $m_\infty(\partial \Omega_\infty)=0$ and $m_\infty(\Omega_\infty)>0$. For each prime ideal $\p\subset \O_K$, let
\[
\Omega_\p\defeq \{x\in \O_{K,\p}^n : x\Mod{\p}\not\in Y(\F_\p)\}.
\]
Then
\begin{align*}
& \#\{x\in \O_K^n \cap B\ast_\w\Omega_\infty : x\Mod{\p} \not\in \Omega_\p \text{ for some prime } \p \text{ with } N_{K/\Q}(\p)>M\} \\
&\qquad =O\left(\frac{B^{|\w|d}}{M^{k-1}\log(M)}\right),
\end{align*}
as a function of $B$, where the implied constant depends only on $\Omega_\infty$ and $Y$. In particular, the limit condition (\ref{eq:KEk}) holds in this situation.
\end{lemma}

\begin{proof}
The case $K=\Q$ is due to Bhargava \cite[Theorem 3.3]{Bha14} (which generalized a result of Ekedahl \cite{Eke91}). Bhargava's method generalizes to arbitrary number fields (as noted in \cite[Theorem 21]{BSW15}) and to the weighted case (as noted in \cite[pg. 4]{BSW22}).
\end{proof}

\section{Counting points on weighted projective stacks}\label{sec:WPROJ}

In this section we prove our results for counting points of bounded height on weighted projective stacks which satisfy prescribed local conditions.

We fix some notation that will be used throughout this section. Let $K$ be a degree $d$ number field over $\Q$ with ring of integers $\O_K$. Let $\Cl_K$ denote the ideal class group of $K$ and let $h_K\defeq \#\Cl_K$ be the class number. Let $\c_1,\dots,\c_{h_K}$ be a set of integral ideal representatives for $\Cl_K$. Let $\varpi_K$ be the number of roots of unity in $K$ and let
\[
 \varpi_{K,\w}\defeq \frac{\varpi_K}{\gcd(w_0,w_1,\dots,w_n,\varpi_K)}.
 \]

Let $\varphi_f:\mcP(\w')\to \mcP(\w)$ be a morphism of $n$-dimensional weighted projective stacks with finite defect (i.e., $f$ has finite defect). Recall that $\mfD_f$ is the set of defects of $f$ (Definition \ref{def:Defect}) and $\widetilde{\mfD}_f$ is the set of prime ideals which divide at least one defect. Define $\widehat{\mfD}_f$ to be the set of ideals of $\O_K$ divisible only by prime ideals in $\widetilde{\mfD}_f$.

A subset $\Omega_v\subseteq \mcP(\w')(K_v)$ will be referred to as a \textbf{(projective) local condition} at the place $v\in \Val(K)$. Let $\Omega_v^{\aff}\subseteq K_v^{n+1}$ denote the affine cone of $\Omega_v$.
For any set of local conditions, $(\Omega_v)_{v\in V}$,  define the sets
\begin{align*}
\Omega &\defeq\{x \in \mcP(\w')(K): x\in \Omega_v \text{ for all } v\in V\},\\
\Omega_\infty &\defeq \{x \in \mcP(\w')(K) : x\in \Omega_v \text{ for all } v\in V\cap \Val_\infty(K)\}.
\end{align*}
Let $\Omega_\infty^{\aff}\defeq \prod_{v\in \Val_\infty(K)} \Omega_v^{\aff}\subseteq K_\infty^{n+1}$. 

For $\a\subset \O_K$ and ideal let
\[
\mu(\a)=\begin{cases}
(-1)^k & \text{ if } \a \text{ is the product of $k$ distinct prime ideals},\\
0 & \text{ if } \a \text{ is divisible by the square of a prime ideal}.
\end{cases}
\]
denote the M\"obius function.

\subsection{Finitely many local conditions}\label{subsec:WPROJfinite}

We will call a finite set of projective local conditions $(\Omega_v)_{v\in V}$ \textbf{admissible} if it satisfies each of the following properties:
\begin{enumerate}[label=(\roman*)]
\item The set $V$ contains all infinite places (i.e., $\Val_\infty(K)\subseteq V$), and $\Omega_\infty^{\aff}$ is definable in $K_\infty^{n+1}$.
\item For each finite place $v\in V\cap \Val_0(K)$, the set $\Omega_v^{\aff}$ is measurable in $K_v^{n+1}$ with boundary of measure zero.
\end{enumerate}

\begin{reptheorem}[\ref{thm:WProjFin}]
Let $(\Omega_v)_{v\in V}$ be an admissible set of projective local conditions with $V=\Val_\infty(K)\sqcup S$. 
Then there exists an explicit constant $\kappa$ such that
\begin{align*}
\#\{x\in \mcP(\w')(K): \Ht_{\w}(\varphi_f(x))\leq B, x\in \Omega\}
= \kappa B^{|\w'|/e(f)}
 + \begin{cases}
O\left(B^{\frac{1}{e(f)}}\log(B)\right) &  {\substack{\text{ if } \w'=(1,1)\\ \text{and}\ K=\Q,}}\\
O\left(B^{\frac{d|\w'|-w'_{\min}}{d e(f)}}\right) & \text{ else, }
\end{cases}
\end{align*}
as a function of $B$. 
\end{reptheorem}

In the proof of the above theorem we will define lattices $\Lambda(\c\n_1)$, a counting function $\# V(\c_i\n_1,\d)$ (which counts certain translates of the lattices $\Lambda(\c\n_1)$), and construct a fundamental domain $\mcF(1)$  for the $\w'$-weighted action of $\O_K^\times$ on $K_\infty^{n+1}$. Then, setting
\[
\kappa'=\frac{m_\infty(\mcF(1))}{\varpi_{K,\w'} |\Delta_K|^{(n+1)/2}} 
m_\infty(\Omega_\infty^{\aff}\cap \mcF(1)) 
\prod_{v\in S} m_v(\Omega^{\aff}_v \cap \O_{K,v}^{n+1}),
\]
an expression for the leading coefficient $\kappa$ is
\[
\kappa=\frac{\kappa'}{\zeta_K(|\w'|)}
\sum_{i=1}^{h_K}\sum_{\d\in \mfD_f}
\sum_{\n_1\in \widehat{\mfD}_f}\mu(\n_1) \frac{\# V(\c_i\n_1,\d) N_{K/\Q}(\c_i)^{|\w'|} N_{K/\Q}(\d)^{\frac{|\w'|}{e(f)}}}{[\O_K^{n+1}:\Lambda(\c_i\n_1)]} \prod_{\p\in \widetilde{\mfD}_f} \frac{1}{1-N_{K/\Q}(\p)^{-|\w'|}}.
\]

\begin{proof}
We have the following partition of the set of $K$-rational points $\mcP(\w')(K)$ into points whose scaling ideals are in the same ideal class:
\[
\mcP(\w')(K)=\bigsqcup_{i=1}^h \{x\in \mcP(\w')(K): [\mfI_{\w'}(x)]=[\c_i]\}.
\] 
This is well defined since the ideal class of a scaling ideal $[\mfI_\w(x)]$ does not depend on the representative of $x$ \cite[Proposition 3.3]{Den98}.

For each $\c\in \{\c_1,\dots,\c_h\}$ consider the counting function
\[
M(\Omega,\c,B)\defeq \#\{x\in \mcP(\w')(K): \Ht_\w(\varphi_f(x))\leq B,\ [\mfI_{\w'}(x)]=[\c],\ x\in \Omega\}.
\]

Consider the (weighted) action of the unit group $\O_K^\times$ on $K^{n+1}-\{0\}$ given by 
\[
u\ast_{\w'}(x_0,\dots,x_n)=(u^{w'_0}x_0,\dots,u^{w'_n}x_n),
\]
 and let $(K^{n+1}-\{0\})/\O_K^\times$ denote the corresponding set of orbits. Let $\Omega^\aff$ be the affine cone of $\Omega$ (i.e., the pullback of $\Omega$ along the map $\A^{n+1}-\{0\}\to \mcP(\w')$). We may describe $M(\Omega,\c,B)$ in terms of $\O_K^\times$-orbits of an affine cone. In particular, there is a bijection between the set
\[
 \{[(x_0,\dots,x_n)]\in (K^{n+1}-\{0\})/\O_K^\times : \Ht_\w(\varphi_f(x))\leq B,\ \mfI_{\w'}(x)=\c,\ x\in \Omega^\aff\}
 \]
and the set
\[
\{[x_0:\cdots:x_n]\in \mcP(\w')(K): \Ht_\w(\varphi_f(x))\leq B,\ [\mfI_{\w'}(x)]=[\c],\ x\in \Omega\},
\]
 given by
 \[
 [(x_0,\dots,x_n)] \mapsto [x_0:\dots:x_n].
 \]
Therefore
\[
M(\Omega,\c,B)=\#\{x\in (K^{n+1}-\{0\})/\O_K^\times : \Ht_\w(\varphi_f(x))\leq B,\ \mfI_{\w'}(x)=\c,\ x\in \Omega^\aff\}.
\]
Recall that the defect of $x$ is defined as 
\[
\delta_f(x)\defeq \frac{\mfI_\w(f(x))}{\mfI_{\w'}(x)^{e(f)}},
\]
and the set of defects is denoted
\[
\mfD_f\defeq \{\delta_f(x) : x\in \mcP(\w')(K)\}.
\]
For ideals $\d\in \mfD_f$ and $\a\subseteq \O_K$, define the counting function
\begin{align*}
&M( \Omega, \a, \d, B)\\
&\quad \defeq \#\{x\in (K^{n+1}-\{0\})/\O_K^\times : \Ht_\w(\varphi_f(x))\leq B,\ \mfI_{\w'}(x)=\a,\ \delta_f(x)=\d,\ x\in \Omega^\aff\}\\
&\quad =\#\left\{x\in (K^{n+1}-\{0\})/\O_K^\times : \frac{|\varphi_f(x)|_{\w,\infty}}{N_{K/\Q}(\a^{e(f)}\d)}\leq B,\ \mfI_{\w'}(x)=\a,\ \delta_f(x)=\d,\ x\in \Omega^\aff\right\}.
\end{align*}
The counting functions we have defined are related as follows:
\[
M(\Omega,\c,B)=\sum_{\d\in \mfD_f} M(\Omega, \c,\d,B).
\]
Our general strategy will be to first find an asymptotic for the counting function
\begin{align*}
& M'(\Omega,\a,\d,B)\\
&\quad \defeq \#\left\{x\in (K^{n+1}-\{0\})/\O_K^\times : \frac{|\varphi_f(x)|_{\w,\infty}}{N_{K/\Q}(\a^{e(f)}\d)}\leq B,\ \mfI_{\w'}(x)\subseteq \a,\ \delta_f(x)=\d,\ x\in \Omega^\aff\right\},
\end{align*}
and then use M\"obius inversion to obtain an asymptotic for $M(\Omega,\a,\d,B)$. This will allow us to give an asymptotic for $M(\Omega,\c,B)$, from which we will obtain Theorem \ref{thm:WProjFin} by summing over the ideal class representatives $\c\in \{\c_1, \c_2,\dots,\c_h\}$.

Define the set
\[
\mcM(\a,\d)\defeq \{x\in K^{n+1}-\{0\} : \mfI_{\w'}(x)\subseteq \a,\ \delta_f(x)=\d\}.
\]
Then
\[
M(\Omega,\a, \d, B)=\#\{x\in \mcM(\a,\d)/\O_K^\times : \Ht_\w(\varphi_f(x))\leq B,\ \mfI_{\w'}(x)=\a,\ x\in \Omega^\aff\}
\]
and
\[
M'(\Omega,\a, \d, B)=\#\{x\in \mcM(\a,\d)/\O_K^\times : |\varphi_f(x)|_{\w,\infty}\leq B N_{K/\Q}(\a^{e(f)}\d),\ x\in \Omega^\aff\}.
\]
The following lemma shows that $\mcM(\a,\d)$ is the disjoint union of finitely many translates of a lattice.

\begin{lemma}
Fix an ideal $\a\subseteq \O_K$. Then there exists a lattice $\Lambda(\a)\subset K^{n+1}_\infty$ such that for any ideal $\d\in \mfD_f$ there exists a finite set $V(\a,\d)\subset \O_K^{n+1}$ for which 
\[
\mcM(\a,\d)\cup\{0\}=\bigsqcup_{v\in V(\a,\d)} (v+\Lambda(\a) ). 
\]
\end{lemma}

\begin{proof}

For each prime ideal $\p\in \widetilde{\mfD}_f$ define
\[
\q\defeq \p^{\max\{v(\d) : \d\in \mfD_f\}}.
\]
For any prime ideal $\p\subset \O_K$ define a lattice in $K_\p^{n+1}$ by
\[
\Lambda_\p(\a)\defeq \begin{cases}
\left(\q \a_\p^{e(f)}\right)^{\w} & \text{ if } \p\in \widetilde{\mfD}_f,\\
\a_\p^{\w'} & \text{ if } \p\not\in \widetilde{\mfD}_f.
\end{cases}
\]
Define sets
\[
\mcM_\p(\a,\d)\defeq \{a\in K_\p^{n+1} : \mfI_{\w'}(a)\subseteq \a_\p,\ \delta_f(a)=\d_\p\},
\]
and note that
\[
	\mcM(\a,\d) = \bigcap_\p \mcM_\p(\a,\d).
\]

If $\p\not\in \widetilde{\mfD}_f$, then $\d_\p=(1)$ and so $\mcM_\p(\a,\d)=\a_\p^{\w'}=\Lambda_\p(\a)$.

If $\p\in \widetilde{\mfD}_f$, then the conditions $\mfI_{\w'}(a)\subseteq \a_\p$ and $\delta_f(a)=\d_\p$ are together equivalent to satisfying the congruence conditions $a_i\equiv 0\mod{\a_\p^{w'_i}}$ and $f_i(a)\equiv 0 \mod{\left(\a_\p^{e(f)}\d_\p\right)^{w_i}}$ for all $i$. 

Possibly re-ordering the elements of $\w$ and $\w'$ we may apply Proposition \ref{prop:w'<e(f)w}, so that $w'_i\leq e(f) w_i$ for all $i$. This, together with the fact that $\d_\p|\q$, shows that the set $\mcM_\p(\a,\d)$ can be determined by finitely many congruence conditions modulo $\Lambda_\p(\a)$. In particular, there exists a finite set $V_{\p}(\a,\d)\subset K_\p^{n+1}$ for which
\[
\mcM_\p(\a,\d)\cup\{0\}=\bigsqcup_{v\in V_\p(\a,\d)} \left(v+\Lambda_\p(\a)\right).
\]
Letting
\[
\Lambda(\a)\defeq \bigcap_\p \left(\Lambda_\p(\a)\cap K^{n+1}\right)=\a^{\w'} \cap \bigcap_{\p\in \widetilde{\mfD}_f} \Lambda_\p(\a)
\]
and
\[
V(\a,\d)\defeq \bigcap_{\p\in \widetilde{\mfD}_f} \left(V_\p(\a,\d)\cap K^{n+1}\right),
\]
we obtain the desired result.
\end{proof}

Note that the set of lattices $(\Lambda(\a))_{\a\subset \O_K}$ is a $\mfD_f$-admissible system of lattices. 

 We are now going to construct a fundamental domain for the $\w'$-weighted action of the unit group $\O_K^{\times}$ on $K_\infty^{n+1}-\{0\}$. This will be done using Dirichlet's Unit Theorem, which we now recall.

\begin{theorem}[Dirichlet's Unit Theorem]
 Let $K$ be a number field with $r_1$ real embeddings and $r_2$ pairs of complex embeddings of $K$ into $\C$. Then the image $\Lambda$ of the map
\begin{align*}
\lambda: \O_K^\times &\to \R^{r_1+r_2}\\
u &\mapsto (\log |u|_v)_{v\in \Val_\infty(K)}
\end{align*}
is a rank $r\defeq r_1+r_2-1$ lattice in the hyperplane $H$ defined by $\sum_{v\in \Val_\infty(K)} x_v=0$, and the kernel of $\lambda$   is the group of roots of unity in $K$, $\varpi(K)$.
\end{theorem}

For each $v\in \Val_\infty(K)$ define a map
\begin{align*}
\eta_v: K_v^{n+1}-\{0\} &\to \R \\
x_v &\mapsto \log\max_i |f_i(x_v)|_v^{1/w_i}.
\end{align*}
Combine these maps to obtain a single map
\begin{align*}
\eta: \prod_{v\in \Val_\infty(K)} (K_v^{n+1}-\{0\}) &\to \R^{r_1+r_2}\\
x & \mapsto (\eta_v(x_v)).
\end{align*}
Let $H$ be the hyperplane in Dirichlet's unit theorem, and let
\begin{align*}
\pr: \R^{r_1+r_2} \to H
\end{align*}
be the projection along the vector $(d_v)_{v\in \Val_\infty(K)}$, where $d_v=1$ if $v$ is real and $d_v=2$ if $v$ is complex. More explicitly,
\[
(\pr(x))_v = x_v-\left(\frac{1}{d}\sum_{v'\in\Val_\infty(K)} x_{v'}\right) d_v.
\]

Let $\{u_1,\dots,u_r\}$ be a basis for the image of the units of $K$, $\O_K^\times$, in the hyperplane $H$, and let $\{\check{u}_1,\dots,\check{u_r}\}$ be the corresponding dual basis. Then the set
\[
\tilde{\mcF}\defeq \{y\in H: 0\leq \check{u}_j(y)<1 \text{ for all }j\in\{1,\dots,r\}\}
\]
is a fundamental domain for the hyperplane $H$ modulo the lattice $\Lambda$, and $\mcF\defeq (\pr\circ\eta)^{-1}\tilde{\mcF}$ is a fundamental domain for the $\w'$-weighted action of the unit group $\O_K^\times$ on the space $\prod_{v\in \Val_\infty(K)}(K_v^{n+1}-\{0\})$. 

Define sets
\[
\mcD(B)\defeq \left\{x\in \prod_{v\in \Val_\infty(K)} (K_v^{n+1}-\{0\}) : |\varphi_f(x)|_{\w,\infty}=\prod_{v\in \Val_\infty(K)} \max_i|f_i(x_{v})|_v^{1/w_i}\leq B\right\}
\]
and $\mcF(B)\defeq \mcF\cap \mcD(B)$. The sets $\mcD(B)$ are $\O_K^\times$-stable, in the sense that if $u\in \O_K^{\times}$ and $x\in \mcD(B)$, then $u\ast_{\w'} x\in \mcD(B)$; this can be seen by the following computation: 
\[
\prod_{v|\infty} \max_i |f_i(u^{w'_i}x_{v})|_v^{1/w_i}=\prod_{v|\infty} |u|_v^{e(f)}\prod_{v|\infty}\max_i |f_i(x_{v})|_v^{1/w_i}=\prod_{v|\infty}\max_i |f_i(x_{v})|_v^{1/w_i}.
\]
 Similarly, for any $t\in \R$, we have that
\[
\prod_{v|\infty} \max_i |f_i(t^{w'_i}x_{v})|_v^{1/w_i}=\prod_{v|\infty} |t|_v^{e(f)}\prod_{v|\infty}\max_i |f_i(x_{v})|_v^{1/w_i}=|t|^{d\cdot e(f)}\prod_{v|\infty}\max_i |f_i(x_{v})|_v^{1/w_i}.
\]
This shows that $|\varphi_f(t\ast_{\w'} x)|_{\w,\infty}=|t|^{e(f)d}|\varphi_f(x)|_{\w,\infty}$. Therefore $\mcD(B)=B^{\frac{1}{e(f) d}}\ast_{\w'}\mcD(1)$ for all $B>0$. On the other hand, the fundamental domain $\mcF$ is stable under the weighted action of $t\in \R^\times$, in the sense that $t\ast_{\w'} \mcF=\mcF$. To see this, note that for any $x\in \prod_{v|\infty} (K_v^{n+1}-\{0\})$, we have that
\[
\eta(t\ast_{\w'} x)=(e(f) d_v)_{v|\infty}\log(|t|)+\eta(x).
\]
Since the projection map $\pr(x)$ is linear and annihilates the vector $(d_v)_{v|\infty}$, we have that
\[
\pr\circ\eta(t\ast_{\w'}x)=\pr\circ\eta(x),
\]
as desired. From our observations we obtain the following lemma:

\begin{lemma}\label{lem:F(B)WHom}
The regions $\mcF(B)$ are weighted homogeneous, in the sense that $\mcF(B)=B^{1/e(f)d}\ast_{\w'} \mcF(1)$ for all $B>0$.
\end{lemma}

We are now going to count lattice points in $\mcF(B)$. In \cite{Sch79} this is done by using the classical Principle of Lipschitz \cite{Dav51} (see also \cite[VI \S 2 Theorem 2]{Lan94}). One of the cruxes of Schanuel's argument is verifying that a fundamental domain (analogous to our $\mcF(1)$) has  Lipschitz parameterizable boundary, so that he can apply the Principle of Lipschitz. Though one can modify this part of Schanuel's argument to work in our case, we will instead take a slightly different route, using an o-minimal version of the Principle of Lipschitz (Proposition \ref{prop:o-Lip}). This allows one to give a more streamlined proof of this part of Schanuel's argument. Techniques from o-minimal geometry have previously been used to study Manin's Conjecture for varieties over number fields (see, e.g., the work of Frei and Pieropan \cite{FP16}).

\begin{lemma}\label{lem:F(1)Bounded}
The set $\mcF(1)$ is bounded.
\end{lemma}

\begin{proof}
Let $\tilde{H}\subset \prod_{v|\infty} K_v$ be the subset defined by $\sum_{v|\infty} x_v\leq 0$. Note that the domain $\mcF(1)$ is equal to $\eta^{-1}(\tilde{H} \cap  \pr^{-1}(\tilde{\mcF}))$. It follows from the definition of $\eta$ that, in order to show $\mcF(1)$ is bounded, it suffices to show that the set 
\[
S\defeq\tilde{H} \cap  \pr^{-1}(\tilde{\mcF})\subset \R^{r_1+r_2}
\]
 is bounded above (i.e., there exist constants $c_1,\dots,c_{r_1+r_2}$ such that for each 
 \[
 x=(x_1,\dots,x_{r_1+r_2})\in S
 \]
  we have $x_i\leq c_i$ for all $i$). For this, note that any $x\in S$ can be written as 
\[
x=\pr(x) + (d_v)_v \frac{1}{d} \sum_{v'|\infty} x_{v'},
\]
and thus the components of $S$ are bounded above, noting that the first term, $\pr(x)$, has components bounded above, and the second term has negative components. 
\end{proof}

\begin{lemma}\label{lem:F(1)Definable}
The set $\mcF(1)$ is definable in $\R_{\exp{}}$.
\end{lemma}

\begin{proof}
We make the following straightforward observations:
\begin{itemize}
\item The set $\mcD(1)$ is semi-algebraic.
\item The subset
\small
\begin{align*}
&\mcF=\{x\in \prod_{v|\infty} (K_v^{n+1}-\{0\}): \pr\circ\eta(x)\in \tilde{\mcF}\}\subset (\R^{r_1+r_2})^{n+1}
\end{align*}
consists of all $x\in \prod_{v|\infty} (K_v^{n+1}-\{0\})$ such that
\begin{align*}
 0\leq \check{u}_j \tiny{\left(\log(\max_i |f_i(x)|_v^{1/w_i})-d_v\left(\frac{1}{d} \sum_{v|\infty}\log(\max_i |f_i(x)|_{v'}^{1/w_i}))\right)\right)_v} <1
\end{align*}
for all $j$.
This set is definable in $\R_{\exp}$, since it can be described in terms of polynomials and $\log$, and $\log$ is definable in $\R_{\exp}$.
\end{itemize}

It follows that the intersection $\mcF(1)=\mcF\cap\mcD(1)$ is definable in $\R_{\exp}$.
\end{proof}

 Recall that $\varpi(K)$ is the group of roots of unity in $K$ and  that $\varpi_K\defeq \#\varpi(K)$ is the number of roots of unity in $K$. Note that each $\varpi(K)$-orbit (with respect to the $\w'$-weighted action) of an element of $(K-\{0\})^{n+1}$ contains 
 \[
 \varpi_{K,\w'}\defeq \frac{\varpi_K}{\gcd(w'_0,\dots,w'_n, \varpi_K)}
 \]
 elements. Thus we may rewrite our counting function as
\begin{align}\label{eq:orbits}
M'(\Omega,\a,\d,B)
=
\frac{\# \left( \Omega^{\aff} \cap \mcF\left(B\cdot N_{K/\Q}(\a^{e(f)}\d)\right)\cap \mcM(\a,\d)\right)}{\varpi_{K,\w'}}
 + O\left(B^{\frac{d|\w'|-w'_{\min}}{e(f)d}}\right),
\end{align}
where the error term accounts for the points counted by $M'(\Omega,\a, \d, B)$ which are contained in the subvariety of the affine cone of $\mcP(\w')(K)$ consisting of points with at least one coordinate equal to zero; by Proposition \ref{prop:KCRTfin} this error is at most $O\left(\left(B^{1/e(f)d}\right)^{d|\w'|-w'_{\min}}\right)$.

For $\a \subseteq \c$ ideals of $\O_K$, write $\a/\c=\n_0\n_1$, where $\n_0$ is the largest factor of $\a/\c$ which is relatively prime to all elements of the set of defects $\mfD_f$.
By Lemma \ref{lem:F(1)Bounded} and Lemma \ref{lem:F(1)Definable} we may apply Proposition \ref{prop:KCRTfin} to the $\#V(\a,\d)$ translates of the lattice $\Lambda(\a)$ and with the local conditions, $\Omega^{\aff}_v\cap\mcF(1)$ if $v|\infty$, and $\Omega^{\aff}_v$ if $v\in S$, to obtain the following asymptotic for $M'(\Omega,\a, \d, B)$ from (\ref{eq:orbits}):
\begin{salign}\label{eq:I<=c asymptotic'}
&\frac{\#V(\a,\d)\cdot m_\infty(\Omega^\aff_\infty\cap \mcF(1))\cdot N_{K/\Q}(\a^{e(f)}\d)^{\frac{|\w'|}{e(f)}}}{\varpi_{K,\w'} N_{K/\Q}(\n_0)^{|\w'|}[\O_K^{n+1} : \Lambda(\c\n_1)]\cdot |\Delta_K|^{(n+1)/2}} \left(\prod_{\p\in S} m_\p(\Omega^{\aff}_\p \cap \O_{K,\p}^{n+1})\right) B ^{\frac{|\w'|}{e(f)}} \\
&\qquad + O\left(B^{\frac{d|\w'|-w'_{\min}}{e(f) d}}\right).
\end{salign}

By the definition of $\kappa'$,  we have
\begin{align}\label{eq:I<=c asymptotic}
M'(\Omega,\a,\d,B)=\kappa' \frac{\#V(\a,\d)\cdot N_{K/\Q}(\c\n_1)^{|\w'|}N_{K/\Q}(\d)^{\frac{|\w'|}{e(f)}}}{[\O_K^{n+1}:\Lambda(\c\n_1)]} B^{\frac{|\w'|}{e(f)}}
+ O\left(B^{\frac{d|\w'|-w'_{\min}}{e(f) d}}\right).
\end{align}
Note that for 
\[
x\in \c^{\w'}=\c^{w'_0}\times \c^{w'_1}\times\cdots\times \c^{w'_n}, 
\]
the scaling ideal $\mfI_{\w'}(x)$ equals $\c\n$ for some ideal $\n\subseteq \O_K$, and thus
\[
\Ht_\w(\varphi_f(x))=\frac{|\varphi_f(x)|_{\w,\infty}}{N_{K/\Q}(\n^{e(f)} \c^{e(f)}\d)}.
\]
 Therefore 
\[
M'(\Omega,\c,\d,B)=\sum_{\n\subseteq \O_K} M(\Omega,\c\n,\d, B/N_{K/\Q}(\n)^{e(f)}).
\]
We now apply M\"obius inversion and use our asymptotic (\ref{eq:I<=c asymptotic}) for $M'(\Omega,\a, \d, B)$ to obtain an asymptotic for the counting function $M(\Omega, \c, \d, B)$:
\begin{align*}
M(\Omega,\c, \d, B) 
&=\sum_{\n\subseteq \O_K} \mu(\n)\left(\kappa'\frac{\#V(\c\n,\d)\cdot N_{K/\Q}(\c\n_1)^{|\w'|}N_{K/\Q}(\d)^{\frac{|\w'|}{e(f)}}}{[\O_K^{n+1}:\Lambda(\c\n_1)]} \left(\frac{B }{N_{K/\Q}(\n)^{e(f)}}\right)^{|\w'|/e(f)}\right.\\
& \left.\hspace{40mm}+ O\left(\left(\frac{B}{N_{K/\Q}(\n)^{e(f)}}\right)^{\frac{d|\w'|-w'_{\min}}{e(f)d}}\right)\right)\\
&=\kappa' B^{|\w'|/e(f)}\sum_{\n\subseteq \O_K}  \mu(\n)\frac{\#V(\c\n,\d)\cdot N_{K/\Q}(\c^{e(f)}\d)^{\frac{|\w'|}{e(f)}}}{N_{K/\Q}(\n_0)^{|\w'|}[\O_K^{n+1}:\Lambda(\c\n_1)]}\\ 
& \hspace{20mm} +O\left(B^{\frac{d|\w'|-w'_{\min}}{e(f)d}} \sum_{\n\subseteq \O_K}\frac{1}{N_{K/\Q}(\n)^{|\w'|-w'_{\min}/d}}\right)\\
&=\kappa' B^{|\w'|/e(f)}\sum_{\n\subseteq \O_K}  \mu(\n)\frac{\#V(\c\n,\d)\cdot N_{K/\Q}(\c^{e(f)}\d)^{\frac{|\w'|}{e(f)}}}{N_{K/\Q}(\n_0)^{|\w'|}[\O_K^{n+1}:\Lambda(\c\n_1)]}\\ 
& \hspace{21mm} 
+ \begin{cases}
O\left(B^{1/e(f)}\log(B)\right) & {\substack{\text{ if } \w'=(1,1)\\ \text{and}\ K=\Q,}}\\
O\left(B^{\frac{d|\w'|-w'_{\min}}{e(f)d}}\right) & \text{ else. }
\end{cases}
\end{align*}

Noting that $\#V(\c\n,\d)=\#V(\c\n_1,\d)$ and letting $\widehat{\mfD}_f$ denote the set of ideals of $\O_K$ divisible only by primes in the set of prime ideals $\widetilde{\mfD}_f$, we have that
\begin{align*}
&\sum_{\n\subseteq \O_K}  \mu(\n)\frac{\#V(\c\n,\d)\cdot N_{K/\Q}(\c^{e(f)}\d)^{\frac{|\w'|}{e(f)}}}{N_{K/\Q}(\n_0)^{|\w'|}[\O_K^{n+1}:\Lambda(\c\n_1)]}\\
&\hspace{1cm} = \left(\sum_{\n_1\in \widehat{\mfD}_f } \mu(\n_1)\frac{\#V(\c\n_1,\d)\cdot N_{K/\Q}(\c^{e(f)}\d)^{\frac{|\w'|}{e(f)}}}{[\O_K^{n+1}:\Lambda(\c\n_1)]}\right)
\left(\sum_{\n_0\not\in \widehat{\mfD}_f}  \mu(\n_0)\frac{1}{N_{K/\Q}(\n_0)^{|\w'|}}\right)\\
&\hspace{1cm} = \left(\sum_{\n_1\in \widehat{\mfD}_f } \mu(\n_1)\frac{\#V(\c\n_1,\d)\cdot N_{K/\Q}(\c^{e(f)}\d)^{\frac{|\w'|}{e(f)}}}{[\O_K^{n+1}:\Lambda(\c\n_1)]}\right)
\frac{1}{\zeta_K(|\w'|)}\prod_{\p\in \widetilde{\mfD}_f} \frac{1}{1-N_{K/\Q}(\p)^{-|\w'|}}.
\end{align*}
Note that the remaining sum is finite since $\mu(\n_1)\neq 0$ only if $\n_1$ is squarefree, and there are only finitely many squarefree elements in $\widehat{\mfD}_f$, since we have assumed that the set of defects $\mfD_f$ is finite. Define the constant
\[
\kappa(\c,\d)\defeq  \frac{\kappa'}{\zeta_K(|\w'|)} \left(\sum_{\n_1\in \widehat{\mfD}_f } \mu(\n_1)\frac{\#V(\c\n_1,\d)\cdot N_{K/\Q}(\c^{e(f)}\d)^{\frac{|\w'|}{e(f)}}}{[\O_K^{n+1}:\Lambda(\c\n_1)]}\right)
\prod_{\p\in \widetilde{\mfD}_f} \frac{1}{1-N_{K/\Q}(\p)^{-|\w'|}}.
\]
Summing over the defects $\d\in \mfD_f$, we have that
\begin{align*}
M(\Omega, \c, B)&=\sum_{\d\in \mfD_f}M(\Omega, \c, \d, B)\\
&= \sum_{\d\in \mfD_f} \kappa(\c,\d) B^{|\w'|/e(f)}
 + \begin{cases}
O\left(B^{1/e(f)}\log(B)\right) & \text{ if } \w'=(1,1) \text{ and } K=\Q,\\
O\left(B^{\frac{d|\w'|-w'_{\min}}{e(f)d}}\right) & \text{ else. }
\end{cases}
\end{align*}
 Finally, summing over the ideal class representatives $\c_i$ of the ideal class group $\Cl_K$, we obtain Theorem \ref{thm:WProjFin}.
\end{proof}

\subsection{Infinitely many local conditions}\label{subsec:WPROJinfinite}

We now address the situation in which infinitely many local conditions are imposed. We will call a full set of local conditions $(\Omega_v)_{v\in \Val(K)}$ \textbf{admissible} if it satisfies each of the following properties:
\begin{enumerate}[label=(\roman*)]
\item For all $v\in \Val(K)$ the set $\Omega_v^{\aff}\subseteq K_v^{n+1}$ is measurable and has boundary $\partial \Omega_v^{\aff}$ of measure zero.
\item For all bounded measurable subsets $\Psi\subset K_\infty^{n+1}$ such that $m_\infty(\Psi)>0$ and $m_\infty(\delta \Psi)=0$, we have that
\begin{align}\label{eq:WPEk}
\lim_{M\to\infty}\limsup_{B\to\infty}\frac{\#\{x\in \O_K^{n+1} \cap (B\ast_{\w'} \Psi) : x\not\in \Omega_v^\aff\ \text{\small for some finite $v$ with } N_{K/\Q}(\p_v)>M\}}{B^{d|\w'|}}=0.
\end{align}
\end{enumerate}

\begin{theorem}\label{thm:WProjInfty}
 Let $(\Omega_v)_{v\in \Val(K)}$ be a full admissible set of local conditions. Then there exists an explicit constant $\kappa$ such that
\begin{align*}
\#\{x\in \varphi_f\left(\mcP(\w')(K)\right):\Ht_{\w}(x)\leq B,\ x\in \varphi_f(\Omega)\} \sim \kappa B^{|{\w}'|/e(f)},
\end{align*}
as a function of $B$. 
\end{theorem}

Maintaining notation from Subsection \ref{subsec:WPROJfinite} and setting
\[
\kappa'=\frac{m_\infty(\mcF(1))}{\varpi_{K,\w'} |\Delta_K|^{(n+1)/2}} 
m_\infty(\Omega_\infty^{\aff}\cap \mcF(1)) 
\prod_{v\in \Val_0(K)} m_{\p_v}(\Omega^{\aff}_{\p_v} \cap \O_{K,\p_v}^{n+1}),
\]
an expression for the leading coefficient in Theorem \ref{thm:WProjInfty} is
\[
\kappa=\frac{\kappa'}{\zeta_K(|\w'|)}\sum_{i=1}^{h_K}\sum_{\d\in \mfD_f} \sum_{\n_1\in \widehat{\mfD}_f} \frac{\# V(\c_i\n_1,\d) N_{K/\Q}(\c_i)^{|\w'|} N_{K/\Q}(\d)^{\frac{|\w'|}{e(f)}}}{[\O_K^{n+1}:\Lambda(\c_i\n_1)]} \prod_{\p\in \widetilde{\mfD}_f} \frac{1}{1-N_{K/\Q}(\p)^{|\w'|}}.
\]

\begin{proof}
The proof is the same as the proof of Theorem \ref{thm:WProjFin}, but with Proposition \ref{prop:KCRTfin} replaced by Proposition \ref{prop:KEk}. In order to apply Proposition \ref{prop:KEk} to give an estimate for $M'(\Omega, \a, \d, B)$ we must check that the local conditions
\begin{align*}
\Theta_\infty &\defeq  \{x\in \Omega_\infty^{\aff} \cap \mcF : |\varphi_f(x)|_{\w,\infty}\leq N_{K/\Q}(\a^{e(f)}\d)\},\\
\Theta_\p &\defeq  \{x\in \Omega_\p^{\aff}\cap \O_{K,\p}^{n+1} : \mfI_{\w'}(x)_\p=\a_\p,\ \delta_f(x)_\p=\d_\p\} 
\end{align*}
satisfy the hypotheses of Proposition \ref{prop:KEk}. 

The sets $\Theta_\p$ are measurable since they are each the finite intersection of measurable sets.
Since $\mcF(1)$ is a bounded definable set by Lemma \ref{lem:F(1)Bounded} and Lemma \ref{lem:F(1)Definable}, it is measurable by Proposition \ref{prop:DefinableMeasurable}. It follows that $\Theta_\infty$ is measurable since it is the intersection of two measurable sets. 

We will now show that the $\Theta_\p$ satisfy condition (\ref{eq:KEk}) of Proposition \ref{prop:KEk}. By the assumption (\ref{eq:WPEk}) it suffices to show (\ref{eq:KEk}) for the sets 
\[
\Theta'_\p\defeq \{x\in \O_{K,\p}^{n+1} : \mfI_{\w'}(x)_\p=\a_\p,\ \delta_f(x)_\p=\d_\p\}.
\] 
Note that for $\p$ sufficiently large 
\[
\Theta'_\p=\{x\in \O_{K,\p}^{n+1} : \mfI_{\w'}(x)_\p=\O_{K,\p},\ \delta_f(x)_\p=\O_{K,\p}\}.
\]
 Applying Lemma \ref{lem:Ek-Bha} with the subscheme corresponding to the ideal $(x_0,\dots,x_n)\cap (f_0(x),\dots,f_n(x))$ shows that the sets
\[
\Theta''_\p\defeq \{x\in \O_{K,\p}^{n+1} : \mfI_{\w'}(x)_\p=\O_{K,\p},\ \mfI_\w(\varphi_f(x))_\p=\O_{K,\p}\}
\]
satisfy condition (\ref{eq:KEk}). But $\Theta''_\p\subseteq \Theta'_\p$ for all sufficiently large $\p$, and thus the $\Theta'_\p$ must also satisfy (\ref{eq:KEk}). 
\end{proof}

In the case that $\w=\w'$ and $f$ is the identity, one may compute the volume of $\mcF(1)$ as in \cite[Proposition 5.3]{Den98} to obtain
\[
m_\infty(\mcF(1))=\left(2^{r_1+r_2}\pi^{r_2}\right)^{n+1}R_K|\w|^{r_1+r_2-1}.
\]
This differs by a factor of $2^{(n+1)r_2}$ from Deng's result since we are using the usual Haar measure on $K_\infty$ rather than the Lebesgue measure. We also have that the set of defects $\mfD_{f}$ is just $\{\O_K\}$ when $f$ is the identity. From these observations, Theorem \ref{thm:WProjFin} and Theorem \ref{thm:WProjInfty} give the following corollary:

\begin{corollary}\label{cor:WPcount-f=id}
Let $(\Omega_v)_{v\in \Val(K)}$ be a full admissible set of local conditions. Define a product of local measures
\[
\kappa_\Omega=m_\infty (\Omega_\infty^{\aff}\cap\ \mcF(1))\prod_\p m_\p(\{x\in \Omega_\p^\aff\cap \O_{K,\p}^{n+1}\}).
\]
Then 
\begin{align*}
&\#\{x\in \mcP(\w)(K):\Ht_\w(x)\leq B,\ x\in \Omega_v \text{ for all } v\in \Val(K)\}\\
&\hspace{7mm}\sim \kappa_\Omega \frac{h_K \left(2^{r_1+r_2}\pi^{r_2}\right)^{n+1}R_K|\w|^{r_1+r_2-1}}{\varpi_{K,\w}|\Delta_K|^{(n+1)/2}\zeta_K(|\w|)}B^{|\w|}.
\end{align*}
Moreover, if $\Omega_\infty$ is definable in $K_\infty^{n+1}$ and $\Omega_v=\mcP(\w)(K_v)$ for all but finitely many $v\in\Val_0(K)$, then
\begin{align*}
&\#\{x\in \mcP(\w)(K):\Ht_\w(x)\leq B,\ x\in \Omega_v \text{ for all } v\in \Val(K)\}\\
&\hspace{15mm} =\kappa_\Omega \frac{h_K \left(2^{r_1+r_2}\pi^{r_2}\right)^{n+1}R_K|\w|^{r_1+r_2-1}}{\varpi_{K,\w}|\Delta_K|^{(n+1)/2}\zeta_K(|\w|)}B^{|\w|}+ \begin{cases}
O\left(B^{|\w|-w_{\min}/d}\log(B)\right) &  {\substack{\text{ if } \w=(1,1)\\ \text{and}\ K=\Q,}}\\
O\left(B^{|\w|-w_{\min}/d}\right) & \text{ else. }
\end{cases}
\end{align*}
\end{corollary}

The next result, which follows immediately from Lemma \ref{lem:Ek-Bha}, gives a criterion for checking the limit condition in Theorem \ref{thm:WProjInfty}:

\begin{lemma}\label{lem:WProj-Ek-Bha}
Let $K$ be a number field and $Y\subset \mcP_{\O_K}(\w')$ a closed subscheme of codimension $k>1$. For each prime ideal $\p\subset \O_K$, let
\[
\Omega_\p=\{x\in \mcP(\w')(\O_{K,\p}) : x\Mod{\p}\not\in Y(\F_\p)\}.
\]
Then, for all bounded subsets $\Psi\subset K_\infty^{n+1}$ with positive measure and boundary of measure zero, 
we have that
\begin{align*}
&\#\{x\in \O_K^{n+1} \cap B\ast_{\w'} \Psi : x\Mod{\p} \not\in \Omega^\aff_\p \text{ for some prime } \p \text{ with } N_{K/\Q}(\p)>M\}\\
&\hspace{1cm}=O\left(\frac{B^{|\w'|d}}{M^{k-1}\log(M)}\right),
\end{align*}
where the implied constant depends only on $\Psi$, $\w$, and $Y$. In particular, the limit condition (\ref{eq:WPEk}) holds in this situation.
\end{lemma}

Let $k$ be a positive integer. An element $\alpha$ in $\O_K$ is said to be \textit{$k$-free} if there are no prime ideals $\p\subset \O_K$ such that $\alpha\in \p^k$. Let $[r]\in K^\times/(K^\times)^k$, let $a\in \O_K$, and let $\n\subset \O_K$ be an integral ideal. If there exists a $k$-free $t\in K^\times$ which is equivalent to $r$ in $K^\times/(K^\times)^k$ and which satisfies $t\equiv a\Mod{\n}$, then we will write $[r]\equiv a\Mod{\n}$. Using Corollary \ref{cor:WPcount-f=id} we can count elements of $\mcP(w_0)(K)$ of bounded height which satisfy a congruence condition.

\begin{corollary}\label{cor:k-free}
Let $K$ be a degree $d$ number field over $\Q$, let $w_0$ be a positive integer, let $a\in \O_K$ be an algebraic integer in $K$, and let $\n\subset \O_K$ be an integral ideal with prime factorization $\n=\p_1^{s_1}\cdots\p_n^{s_n}$. Define the ideal $\g\defeq \gcd((a),\n)$. Then the number of elements $x\in \mcP(w_0)(K)=K^\times/(K^\times)^{w_0}$ satisfying $x\equiv a\Mod{\n}$ and of height less than $B$ is
\[
\left(\prod_{\substack{\p_i|\n \\ \p_i| \g}}N_{K/\Q}(\p_i)^{-s_i}\prod_{\substack{\p_i|\n \\ \p_i\nmid \g}} \frac{N(\p_i)^{w_0-s_i}}{N_{K/\Q}(\p_i)^{w_0}-1}\right) \frac{h_K 2^{r_1+r_2}\pi^{r_2}R_K w_0^{r_1+r_2-1}}{\varpi_{K,w_0}|\Delta_K|^{1/2}\zeta_K(w_0)}B^{w_0}+
O\left(B^{w_0(1-1/d)}\right).
\]
\end{corollary}

\begin{remark}
In the case that $K$ has class number $1$, this result counts the number of $w_0$-free integers of $K$ satisfying a congruence condition.
\end{remark}

\begin{proof}
As $x\in K^\times/(K^\times)^{w_0}$ and $\n=\p_1^{s_1}\cdots\p_n^{s_n}$, the congruence $x\equiv a \Mod{\n}$ only depends on the $s_i$ modulo $w_0$. We may thus assume $s_i<w_0$ for all $i$.

First consider the case $\g=(1)=\O_K$. We apply Corollary \ref{cor:WPcount-f=id} with local conditions
\[
\Omega_{v}\defeq\begin{cases}
\mcP(w_0)(K_v) 
& \text{ if } v|\infty \text{ or } v(\n)=0, \\
  \{x\in\mcP(w_0)(K_{v}) : x \equiv a \Mod{\p_i^{s_i}}\} & \text{ if } \p_v=\p_i \text{ for some $i$}.
  \end{cases}
\]
As
\[
\Omega^{\aff}_{\p_i}=\{x\in K_{\p_i}^\times : x\not\equiv 0 \Mod{\p_i^{w_0}},\ x \equiv a \Mod{\p_i^{s_i}}\},
\]
we have that
\[
m_{\p_i}(\Omega_{\p_i}^{\aff}\cap \O_{K,\p_i})=\frac{N_{K/\Q}(\p_i)^{w_0-s_j}}{N_{K/\Q}(\p_i)^{w_0}-1},
\]
since $N_{K/\Q}(\p_i)^{w_0-s_j}$ of the $N_{K/\Q}(\p_i)^{w_0}-1$ non-zero elements of $\O_{K,\p_i}/\p_i^{w_0}$ are equivalent to $a$ modulo $\p_i^{s_i}$. From this the desired result follows in the case $\g=(1)$.

If $\g\neq (1)$, then the desired result is obtained by first replacing $\n$ by $\n/\g$, then carrying out the same argument as in the $\g=(1)$ case, and finally multiplying by a factor of $N_{K/\Q}(\g)^{-1}=\prod_{\substack{\p_i|\n \\ \p_i| \g}} N_{K/\Q}(\p_i)^{-s_i}$ to account for the $1$ out of $N_{K/\Q}(\g)$ elements $x$ of $\O_K^\times$ within a single congruence class modulo $\g$.
\end{proof}




\subsection{Counting with a twisted height}\label{subsec:WPROJtwisted}

In this section, motivated by counting twists of elliptic curves, we count \textit{twists} of points on weighted projective stacks.

Let $\w=(w_0,\dots,w_n)\in \Z_{>0}^{n+1}$ and let $\tau$ be a positive integer which divides $\gcd(w_0,\dots,w_n)$. Let $t\in \mcP(\tau)(K)=K^\times/(K^\times)^\tau$. For $x=[x_0:\cdots:x_n]\in \mcP(\w)(K)$ we define the \textbf{twist}, $x^{(t)}$ of $x$ in $\mcP(\w)$ by $t$, to be
\[
x^{(t)}\defeq [t^{w_0/\tau}x_0:\dots:t^{w_n/\tau}x_n]\in \mcP(\w)(K).
\]

\begin{theorem}\label{thm:WProjTwisted}
Let $\tau|\gcd(w_0,\dots,w_n)$ be a positive integer such that
\[
\begin{cases}
-3\tau+|\w| \geq 0 & \text{ if } \tau\geq |\w'|/e(f),\\
\tau(|\w|-|\w'|/e(f)-1)-|\w|\geq 0 & \text{ if } \tau<|\w'|/e(f).
\end{cases}
\] 
 For ideals $\a,\b\subseteq \O_K$ with $\b$ $\tau$-free we write $\a|||\b$ if $\p^{s}||\a$ implies $\p^{s}||\b$ for all $\p|\a$. Let $\kappa_f$ denote the leading coefficient from Theorem \ref{thm:WProjInfty} when only trivial local conditions are imposed, and let
\[
\kappa=\kappa_f \frac{h_K 2^{r_1+r_2}\pi^{r_2}R_K \tau^{r_1+r_2-1}}{\varpi_{K,\tau}|\Delta_K|^{1/2}\zeta_K(\tau)}.
\]
Let $\lambda(\b)$ be the following product of local densities (to be defined in the proof):
\[
\lambda(\b)\defeq \prod_{\p} m_\p( \Omega(\b)_\p^\aff\cap \O_{K,\p}^{n+1}).
\]
Let $\b=\prod \p_v^{b_v}$ be the prime factorization of $\b$. For $s\in\Q$, let
\[
\sigma(\b,s)\defeq \sum_{\a|||\b} \left(N_{K/\Q}(\sqrt{\a})^{s}\prod_{\p|\a} \frac{1}{N_{K/\Q}(\p_v^{b_v})} \prod_{\p|\b \text{ and } \p\nmid \a}\left( 1-\frac{1}{N_{K/\Q}(\p_v^{b_v})}\right)\right),
\]
where the products are over prime ideals $\p$ satisfying the given divisibility conditions.
 Then
\begin{equation}
\begin{aligned}\label{eq:WProjTwisted}
&\#\{(x,t)\in \mcP(\w')(K)\times \mcP(\tau)(K) : \Ht_{\w}(\varphi_f(x)^{(t)})\leq B\}\\
&\hspace{1.5cm}\sim
\begin{cases}
\left(\frac{\tau}{|\w'|/e(f)-\tau}\kappa\sum\limits_{\substack{\b\subseteq \O_K\\ \tau-\text{free}}} \lambda(\b) \sigma(\b,|\w'|/\tau)\right)B^{|\w'|/e(f)}  & \text{ if } \tau < |\w'|/e(f),\\
\left(\tau \kappa\sum\limits_{\substack{\b\subseteq \O_K\\ \tau-\text{free}}} \lambda(\b) \sigma(\b,\tau) \right)B^{\tau}\log(B) & \text{ if } \tau=|\w'|/e(f), \\
\left(\frac{|\w'|}{e(f)\tau-|\w'|}\kappa\sum\limits_{\substack{\b\subseteq \O_K\\ \tau-\text{free}}} \lambda(\b) \sigma(\b,\tau) \right)B^{\tau} & \text{ if } \tau > |\w'|/e(f).
\end{cases}
\end{aligned}
\end{equation}
\end{theorem}

Before proving Theorem \ref{thm:WProjTwisted}, we state the product lemma \cite[Lemma 7.1]{Den98} (similar results can be found in \cite[Lemma 3.1, Lemma 3.2]{Wan21} and \cite[Proposition 2]{FMT89}).

\begin{lemma}[Product Lemma]\label{lem:product}
Let $C_X$ and $C_Y$ be counting functions on the sets $X$ and $Y$ respectively. Suppose that
\[
F_X(B)\defeq \#\{x\in X: C_X(x)\leq B\} \sim c_X B^{\alpha_X}
\]
and 
\[
F_Y(B)\defeq \#\{y\in Y: C_Y(y)\leq B\} \sim c_Y B^{\alpha_Y}.
\]
Then
\[
F(B)\defeq \#\{(x,y)\in X\times Y: C_X(x)\cdot C_Y(y)\leq B\} \sim \begin{cases}
\frac{c_X c_Y \alpha_X}{\alpha_Y-\alpha_X} B^{\alpha_Y} & \text{ if } \alpha_X<\alpha_Y,\\
c_X c_Y \alpha_X B^{\alpha_X}\log(B) & \text{ if } \alpha_X=\alpha_Y, \\
\frac{c_X c_Y \alpha_Y}{\alpha_X-\alpha_Y} B^{\alpha_X} & \text{ if } \alpha_X>\alpha_Y.
\end{cases}
\]
\end{lemma}

\begin{proof}[Proof of Theorem \ref{thm:WProjTwisted}.]
For $x\in \varphi_f\left(\mcP({\bf{w'}})(K)\right)$ and $v\in \Val_0(K)$, let $\pi_v$ be a uniformizer at $v$ and define a non-negative integer $b_{x,v}$ as follows:
\begin{align*}
b_{x,v}\defeq \min \{b\in \{1,\dots, \tau-1\} : |x^{(\pi_v^b)}|_{{\w},v} < |x|_{{\w},v}\}, 
\end{align*}
where we define the minimum of the empty set to be zero, and where $x^{(\pi_v^b)}$ is the twist of $x$ by $\pi_v^b$ as defined in Subsection \ref{subsec:WPSheights}. Let $\b_x$ be the integral ideal
\[
\b_x=\prod_{v\in\Val_0(K)} \p_v^{b_{x,v}}.
\]
Let $\b\subseteq \O_K$ be any $\tau$-free integral ideal, i.e., the prime factorization 
\[
\b=\prod_{v\in \Val_0(K)}\p_v^{b_v}
\]
 of $\b$ is such that $b_v<\tau$ for all $v$. Let $V_\b\defeq \{v\in \Val_0(K): \p_v|\b \}$. Let $\a$ be an ideal dividing $\b$ such that there exists a subset of places $V_\a\subset \Val_0(K)$ for which
\[
\a=\prod_{v\in V_\a} \p_v^{b_v}.
\]

For each $t\in \mcP(\tau)(K)$ let $\t\subseteq \O_K$ denote the integral ideal generated by the non-zero elements $t'\in \O_K-\{0\}$ for which $[t']=t$ in $\mcP(\tau)(K)$.
For $\b$ a $\tau$-free ideal and $\a|||\b$, define the set
\[
\mcA(\b,\a)\defeq \left\{(x,t)\in \mcP({\w}')(K)\times \mcP(\tau)(K) : \b_{\varphi_f(x)}=\b,\ \p_v^{b_v}|\t\  \forall v\in V_\a,\ \p_v^{b_v}\nmid \t\ \forall v\in V_\b-V_\a\right\}.
\]

Observe that
\[
\prod_{v\in V} |\pi_v|_v^{-1}=\prod_{v\in V} [\O_{K,v}:\pi_v \O_{K,v}]=N_{K/\Q}(\sqrt{\a}),
\]
where $\sqrt{\a}$ is the radical ideal of $\a$. Then, for any $(x,t)\in \mcA(\b,\a)$, we have
\begin{align*}
\Ht_{\w}(x^{(t)})=\frac{N_{K/\Q}(\t)^{1/\tau}\Ht_{\w}(x)}{N_{K/\Q}(\sqrt{\a})}.
\end{align*}

Therefore
\begin{align*}
&\#\{(x,t)\in \mcP(\w')(K)\times \mcP(\tau)(K) : \Ht_{\w}(\varphi_f(x)^{(t)})\leq B\}\\
&\hspace{1cm}=\sum_{\substack{\b\subseteq \O_K\\ \tau\text{--free}}}\sum_{\a|||\b} \#\{(x,t)\in \mcA(\b,\a) : N_{K/\Q}(\t)^{1/\tau}\Ht_{\w}(x)\leq N_{K/\Q}(\sqrt{\a}) B\}.
\end{align*}
Our strategy will be to count each summand on the right using the product lemma. 

Let $z\in \varphi_f\left(\mcP({\w}')(K)\right)$ and let $[z_0:\cdots:z_n]$ be a representative of $z$ which is integral (i.e., $z_i\in \O_K$ for all $i$).

By the definition of $|\cdot|_{\w,v}$, the inequality $|z^{(\pi_v^m)}|_{{\w},v} < |z|_{{\w},v}$ is equivalent to 
\[
\min_j \left\{ \left\lfloor \frac{v(\pi_v^{w_jm/\tau}z_j)}{w_j}\right\rfloor\right\} > \min_j \left\{\left\lfloor \frac{v(z_j)}{w_j}\right\rfloor\right\}.
\]
Letting $s$ be an index such that $v(z_s)/w_s=\min_j\{v(z_j)/w_j\}$, it follows that
\[
b_{z,v}=\begin{cases}
0 & \text{ if } v(z_s)/w_s-\lfloor v(z_s)/w_s\rfloor < 1/\tau\\
\left\lceil \tau\left(1+\lfloor \frac{v(z_s)}{w_s}\rfloor-\frac{v(z_s)}{w_s}\right)\right\rceil & \text{ otherwise.}
\end{cases}
\]

This allows us to translate the condition $\b_z=\b=\prod_v \p_v^{b_v}$ into the following set of local conditions, 
\begin{align*}
\Omega(\b)_v\defeq \begin{cases}
\left\{x\in \mcP({\w}')(K_v): v(f_s(x))-w_s\left\lfloor \frac{v(f_s(x))}{w_s}\right\rfloor < \frac{w_s}{\tau} \right\} & \text{ if } b_v=0,\\
 \left\{x\in \mcP({\w}')(K_v): \frac{w_s(\tau-b_v)}{\tau}\leq v(f_s(x))-w_s\left\lfloor \frac{v(f_s(x))}{w_s}\right\rfloor <\frac{w_s(\tau-b_v+1)}{\tau} \right\} & \text{ if } b_v\neq 0.
\end{cases}
\end{align*}

For almost all $v$ (namely those with $b_v=0$), we have the inclusion of local conditions $\{x\in \O_{K,\p_v}^{n+1} : \mfI_\w(\varphi_f(x))_{\p_v}=\O_{K,\p_v}\}\subset \Omega(\b)_v^{\aff}$. As we saw in the proof of Theorem \ref{thm:WProjInfty}, it follows from Lemma \ref{lem:Ek-Bha} that the local conditions $\{x\in \O_{K,\p}^{n+1} : \mfI_\w(\varphi_f(x))_\p=\O_{K,\p}\}$ satisfy the limit condition (\ref{eq:WPEk}) of Theorem \ref{thm:WProjInfty}. Therefore the local conditions $\Omega(\b)$ satisfy the conditions of Theorem \ref{thm:WProjInfty}. By Theorem \ref{thm:WProjInfty} we obtain the following asymptotic:
\begin{align}\label{eq:b-count}
\#\{x\in \varphi_f\left(\mcP({\w}')(K)\right):\b_x=\b,\ \Ht_{\w}(x)\leq B\}
\sim \kappa_f \lambda(\b) B^{\frac{|\w'|}{e(f)}}. 
\end{align}

By Corollary \ref{cor:k-free}, we have that
\begin{equation}\begin{aligned}\label{eq:a-count}
&\#\{t\in \mcP(\tau)(K): \p_v^{b_v}|\t\ \forall v\in V_\a,\ \p_v^{b_v}\nmid \t\ \forall v\in V_\b-V_\a,\ N(\t)^{1/\tau}\leq B\}\\
&\hspace{1cm}\sim \frac{h_K 2^{r_1+r_2}\pi^{r_2}R_K \tau^{r_1+r_2-1}}{\varpi_{K,\tau}|\Delta_K|^{1/2}\zeta_K(\tau)} \prod_{v\in V_\a} \frac{1}{N_{K/\Q}(\p_v^{b_v})} \prod_{v\in V_\b-V_\a}\left( 1-\frac{1}{N_{K/\Q}(\p_v^{b_v})}\right)B^\tau.
\end{aligned}
\end{equation}

Applying the Product Lemma (Lemma \ref{lem:product}) with the asymptotics (\ref{eq:a-count}) and (\ref{eq:b-count}), and summing over all $\a$ and $\b$ gives the asymptotic (\ref{eq:WProjTwisted}).


We now check that the leading coefficients in (\ref{eq:WProjTwisted}) converge. Assume $\tau\geq |\w'|/e(f)$. In these cases it suffices to show that the sum
\begin{align}\label{eq:infinite_sum}
\sum\limits_{\substack{\b\subseteq \O_K\\ \tau-\text{free}}} \lambda(\b) \sigma(\b,\tau)
\end{align}
converges.

Applying a bound for the growth rate of the sum-of-divisors function (Theorem \ref{thm:bound-sum-of-divisors}) to $(\sqrt{\b})^\tau/\b$, which makes sense because $\b$ is $\tau$-free, we obtain an upper bound for $\sigma(\b,\tau)$:
\begin{align*}
\sigma(\b,\tau)
&=\sum_{\a|||\b} \left(N_{K/\Q}(\sqrt{\a})^{\tau}\prod_{\p_v|\a} \frac{1}{N_{K/\Q}(\p_v^{b_v})} \prod_{\p_v|\b \text{ and } \p_v\nmid \a}\left( 1-\frac{1}{N_{K/\Q}(\p_v^{b_v})}\right)\right)\\
&\leq \sum_{\a|||\b} \left(N_{K/\Q}(\sqrt{\a})^{\tau}\cdot \frac{1}{N_{K/\Q}(\a)} \cdot 1 \right)\\
&=\sum\limits_{\a|||\b} \frac{N_{K/\Q}(\sqrt{\a})^\tau}{N_{K/\Q}(\a)}\\
&=O\left( \frac{N_{K/\Q}(\sqrt{\b})^\tau}{N_{K/\Q}(\b)} \log\log\left(\frac{N_{K/\Q}(\sqrt{\b})^\tau}{N_{K/\Q}(\b)}\right)\right).
\end{align*}
To estimate the product of local measures $\lambda(\b)$ we first estimate the measures $m_v(\Omega(\b)_v)$ for $\p_v|\b$. From the definition of the local conditions $\Omega(\b)_v$ we obtain
\[
m_v(\Omega(\b)_v)=O\left(N_{K/\Q}(\p_v)^{\frac{-|\w|(\tau-b_v)}{\tau}}\right).
\]
Taking the product over $\p_v|\b$, we have the upper bound
\[
\lambda(\b)=O\left(\prod_{\p_v|\b}N_{K/\Q}(\p_v)^{\frac{-|\w|(\tau-b_v)}{\tau}}\right)=O\left(\left(\frac{N_{K/\Q}(\sqrt{\b})^\tau}{N_{K/\Q}(\b)}\right)^{-|\w|/\tau}\right).
\]
It follows that
\begin{align*}
\lambda(\b) \sigma(\b,\tau)
&=O\left(\left(\frac{N_{K/\Q}(\sqrt{\b})^\tau}{N_{K/\Q}(\b)}\right)^{1-|\w|/\tau}\log\log\left(\frac{N_{K/\Q}(\sqrt{\b})^\tau}{N_{K/\Q}(\b)}\right)\right)\\
&=O\left(\left(\frac{N_{K/\Q}(\sqrt{\b})^\tau}{N_{K/\Q}(\b)}\right)^{-2}\log\log\left(\frac{N_{K/\Q}(\sqrt{\b})^\tau}{N_{K/\Q}(\b)}\right)\right),
\end{align*}
where we have used the assumption that $-3\tau+|\w|\geq 0$ in obtaining the second equality. Note that the map
\begin{align*}
\{\b\subseteq \O_K: \b \text{ is } \tau\text{-free}\} &\to \{\b\subseteq \O_K: \b \text{ is } \tau\text{-free}\}\\
\b=\prod_{\p_v|\b} \p_v^{b_v}&\mapsto \frac{N_{K/\Q}(\sqrt{\b})^\tau}{N_{K/\Q}(\b)}=\prod_{\p_v|\b} \p_v^{\tau-b_v}
\end{align*}
is a bijection. Therefore
\[
\sum\limits_{\substack{\b\subseteq \O_K\\ \tau-\text{free}}} \lambda(\b) \sigma(\b,\tau)=\sum\limits_{\substack{\b\subseteq \O_K\\ \tau-\text{free}}} O\left(N_{K/\Q}(\b)^{-2}\log\log(N_{K/\Q}(\b))\right)=O(1),
\]
where the implied constant is absolute.
This shows that the sum (\ref{eq:infinite_sum}) converges. A similar argument shows that the infinite sum in the $\tau<|\w'|/e(f)$ case converges provided that 
\[
\tau(|w|-|\w'|/e(\varphi)-1)-|\w|\geq 0.
\]
\end{proof}

\begin{remark}
It is interesting to note the similarity of the proof of Theorem \ref{thm:WProjTwisted} with the work of Wang on Malle's Conjecture for $S_n\times A$ extensions \cite{Wan17,Wan21}. One of the key features of Wang's work is dealing with the discriminant of the compositum of number fields, which is roughly the product of the discriminants, but with some `defect' making the problem more difficult. The situation of our Theorem \ref{thm:WProjTwisted} is similar in that the twisted height is roughly the product of $N_{K/\Q}(\t)^{1/\tau}$ and $\Ht_\w(x)$, but with the `defect' $N_{K/\Q}(\sqrt{\a})$.
\end{remark}

We will apply Theorem \ref{thm:WProjTwisted} to modular curves isomorphic to $\mcP(2,2)$. The case of interest will be when $\tau=2$, $\w=(1,1)$, and $\w'=(4,6)$. Although $\mcP(2,2)$ is not isomorphic to  $\mcP(1,1)\times \mcP(2)$ (they have different Picard groups), there is an explicit bijection of $K$-points
\begin{salign}\label{eq:theta_bijection}
\theta:\mcP(1,1)(K)\times \mcP(2)(K) &\to \mcP(2,2)(K)\\
([x:y],t) &\mapsto [tx : ty].
\end{salign} 
Define $\theta^{-1}([tx:ty])_1 \defeq [x:y]$ and $\theta^{-1}([tx:ty])_2 \defeq t$.


For future reference, we record the following special case of Theorem \ref{thm:WProjTwisted}.

\begin{corollary}\label{cor:WProjTwisted}
 Let $\a,\b\subseteq \O_K$ be ideals with $\b$ square-free. Let $\kappa_f$ denote the leading coefficient from Theorem \ref{thm:WProjInfty} when only trivial local conditions are imposed, let
\[
\kappa=\kappa_f \frac{h_K 2^{2(r_1+r_2)}\pi^{r_2}R_K}{\varpi_{K}|\Delta_K|^{1/2}\zeta_K(2)},
\]
and let $\lambda(\b)$ be the following product of local densities (as defined in the proof of Theorem \ref{thm:WProjTwisted}):
\[
\lambda(\b)\defeq \prod_{\p} m_\p( \Omega(\b)_\p^\aff\cap \O_{K,\p}^{2}).
\]
Let
\[
\sigma(\b,2)= \sum_{\a|\b} \left(N_{K/\Q}(\sqrt{\a})^{2}\prod_{\p|\a} \frac{1}{N_{K/\Q}(\p_v)} \prod_{\p|\b \text{ and } \p\nmid \a}\left( 1-\frac{1}{N_{K/\Q}(\p_v)}\right)\right),
\]
and let $\theta$ be the bijection (\ref{eq:theta_bijection}).
 Then
\begin{equation}
\begin{aligned}
&\#\{(x\in \mcX : \Ht_{\w}\left(\varphi_f(\theta^{-1}(x)_1)^{(\theta^{-1}(x)_2)}\right)\leq B\}\\
&\hspace{1.5cm}\sim
\begin{cases}
\left(2 \kappa\sum\limits_{\substack{\b\subseteq \O_K\\ \text{square-free}}} \lambda(\b) \sigma(\b,2) \right)B^{2}\log(B) & \text{ if } e(f)=1, \\
\left(\frac{\kappa}{e(f)-1}\sum\limits_{\substack{\b\subseteq \O_K\\ \text{square-free}}} \lambda(\b) \sigma(\b,2) \right)B^{2} & \text{ if } e(f)>1.
\end{cases}
\end{aligned}
\end{equation}
\end{corollary}

\section{Counting elliptic curves}\label{sec:EC}

We now apply the results of the previous section to count elliptic curves with prescribed level structures.

\subsection{Counting elliptic curves with prescribed level structure}\label{subsec:CountingEC}

For $G$ a subgroup of $\GL_2(\Z/N\Z)$ let $N_{\GL_2(\Z/N\Z)}(G)$ denote its normalizer in $\GL_2(\Z/N\Z)$. Let $I\in \GL_2(\Z/N\Z)$ denote the identity matrix. Define
\[
r(G)\defeq\begin{cases}
[N_{\GL_2(\Z/N\Z)}(G):G] & \text{ if } -I\in G,\\
[N_{\GL_2(\Z/N\Z)}(G):G]/2 & \text{ if } -I\not\in G.
\end{cases}
\]

We will need the following lemma, which was proven by Cullinan, Kenney, and Voight \cite[Lemma 3.1.8]{CKV22}. They proved their result for elliptic curves over the rational numbers, however the proof works over arbitrary number fields.

\begin{lemma}\label{lem:fiber}
Let $G$ be a subgroup of $\GL_2(\Z/N\Z)$ and let $\psi:\mcX_G\to \mcX_{\GL_2(\Z)}$ be the morphism that forgets the level structure. Then, for any elliptic curve $E$ over a number field $K$ with automorphism group $\Aut_{\overline{K}}(E)\cong \Z/2\Z$ and which admits a $G$-level structure, there exists a subgroup $G_E\leq G$ for which
\[
\# \psi^{-1}(E)=r(G_E),
\]
where $\# \psi^{-1}(E)$ denotes the  cardinality of the fiber $\psi^{-1}(E)$.
\end{lemma}

We now prove Theorem \ref{thm:EllipticCount}.

\begin{reptheorem}[\ref{thm:EllipticCount}]
Let $G$ be a subgroup of $\GL_2(\Z/N\Z)$.
 Let $K$ be a degree $d$ number field which contains $K_G$. Suppose that for some pair of positive integers $\w'=(w'_0,w'_1)$ there is an isomorphism of $K_G$-stacks from $\mcX_G$ to $\mcP(\w')$. 
Assume that $(w_0, w_1)=1$ or $e(\varphi_G))=1$.
Then there exists an explicit constant $\kappa_G$ such that the number of isomorphism classes of elliptic curves over $K$ admitting a $G$-level structure and which have naive height less than $B$ is
\begin{align*}\label{eq:thm_statment}
\kappa_G B^{\frac{|\w'|}{12 e(G)}}+O\left(B^{\frac{d|\w'|-w'_{\min}}{12 e(G) d}}\right).
\end{align*}
Defining
\[
\kappa'=\frac{m_\infty(\mcF(1))}{r(G) \varpi_{K,\w'}|\Delta_K|},
\]
the leading coefficient can be expressed as
\[
\kappa_G=\frac{\kappa'}{\zeta_K(|\w'|)}
\sum_{i=1}^{h_K}\sum_{\d\in \mfD_f} \sum_{\n_1\in \widehat{\mfD}_f} \frac{\# V(\c_i\n_1,\d) N_{K/\Q}(\c_i)^{|\w'|} N_{K/\Q}(\d)^{\frac{|\w'|}{12 e(G)}}}{[\O_K^{2}:\Lambda(\c_i\n_1)]} \prod_{\p\in \widetilde{\mfD}_f} \frac{1}{1-N_{K/\Q}(\p)^{|\w'|}}.
\]
\end{reptheorem}

\begin{proof}
We first count isomorphism classes of elliptic curves \textit{equipped} with a $G$-level structure. To do this we apply Theorem \ref{thm:WProjFin} with $\mcP(\w')=\mcX_G$ and $\mcP(\w)=\mcP(4,6)\cong \mcX_{\GL_2(\Z)}$ and trivial local conditions $\Omega_\p=\mcP(\w')(K_\p)$ for all $\p$. Let $\psi:\mcX_G\to \mcX_{\GL_2(\Z)}$ be the map which forgets the level structure and let $\varphi:\mcP(\w')\to\mcP(4,6)$ be the corresponding map of weighted projective stacks. Note that $\varphi$ has finite defect by Proposition \ref{prop:finite-defect}.
The degree of $\varphi$ is
 \begin{equation}\label{eq:wps_degree}
 	\deg(\varphi)=e(\varphi)\frac{\lcm(4,6)}{\lcm(w'_0,w'_1)} =  e(\varphi)\frac{12}{\lcm(w'_0,w'_1)}=
 \begin{cases}
 e(\varphi) \frac{24}{w_0',w_1'} & \text{ if } \gcd(w_0',w_1')=2,\\
 	e(\varphi) \frac{12}{w_0',w_1'} & \text{ if } \gcd(w_0',w_1')=1,\\
 \end{cases}
 \end{equation}
 where the cases correspond to whether the generic residual gerbe of $\mcP(w_0', w_1')$ is $\mu_2$ (when $\gcd(w_0',w_1')=2$) or trivial (when $\gcd(w_0',w_1')=1$).
  The degree of $\psi$ is  
\begin{equation}\label{eq:modular_degree}
\deg(\psi)=
\begin{cases}
[\GL_2(\Z/N\Z):G]=[\SL_2(\Z):\Gamma_G] & \text{ if } -I\in G,\\
[\GL_2(\Z/N\Z):G]/2=[\SL_2(\Z):\Gamma_G]/2 & \text{ if } -I\notin G,
\end{cases}
\end{equation}
 where the cases correspond to whether the generic residual gerbe of $\mcX_G$ is $\mu_2$ (when $-I\in G$) or trivial (when $-I\notin G$).  As $\deg(\varphi)=\deg(\psi)$, equating (\ref{eq:wps_degree}) and (\ref{eq:modular_degree}) shows that
 \[
e(G)=e(\varphi)=\frac{w'_0 w'_1}{24}[\SL_2(\Z):\Gamma_G].
\]
 
Taking into account that the naive height on $\mcX_{\GL_2(\Z)}$ is the twelfth power of the height $\Ht_\w$ on $\mcP(4,6)$, Theorem \ref{thm:WProjFin} gives the following asymptotic for the number of isomorphism classes of elliptic curves equipped with a $G$-level structure:
\[
r(G)\kappa_G B^{\frac{|\w'|}{12 e(G)}}+O\left(B^{\frac{d|\w'|-w'_{\min}}{12 e(G) d}}\right).
\]

To refine this to a count of isomorphism classes of elliptic curves that \textit{admit} a $G$-level structure, we use Lemma \ref{lem:fiber}. For $B\in \R_{\geq 0}$ let $S(B)$ be the set of elliptic curves over $K$ which admit a $G$-level structure and are of height bounded by $B$. By Theorem \ref{thm:WProjFin} and Lemma \ref{lem:fiber}, to prove the theorem it will suffice to show that 
\begin{align}\label{eq:admit_level_structure}
 \#\{E \in S(B) : \#\psi^{-1}(E)\neq r(G)\}=o\left(B^{\frac{|\w'|}{12 e(G)}}\right). 
 \end{align}
 
If $\Aut_{\overline{K}}(E)\not\cong \Z/2\Z$, then $\Aut_{\overline{K}}(E)\in \{\Z/4\Z, \Z/6\Z\}$. Suppose $\Aut_{\overline{K}}(E)\cong \Z/6\Z$. Then $E$ has $j$-invariant zero, and can be written as $E:y^2=x^2+B$. 
By Theorem \ref{thm:WProjMorphisms}, we know that $\psi$ can be given by polynomials, say $\psi_1,\psi_2\in K[x,y]$. Note that in order for the elliptic curve 
\[
E_{a,b}: y^2 = x^3 + \psi_1(a,b) x + \psi_2(a,b)
\]
to have $j$-invariant zero, $\psi_1(a,b)$ must equal $0$. Thus, for any $m\in \Z$, each fiber $(E,\rho)\in \mcX_G$ above $E$  satisfies the local condition 
\[
\Omega_v^{\aff}(m)\defeq \{(a,b)\in K_v^2 : v(\psi_{1}(a,b))\geq m\}.
\]
These sets are measurable, have boundary of measure zero, and
\[
m_v(\Omega_v^{\aff}(m)\cap \O_{K,v}^{2})\to 0 \quad \text{ as  }\quad m\to\infty.
\]
 Therefore, by Theorem \ref{thm:WProjFin}, the number of elliptic curves with $G$-level structure and $\Aut_{\overline{K}}(E)\cong \Z/6\Z$ is $O\left(B^{\frac{d|\w'|-w'_{\min}}{12 e(G) d}}\right)$. An analogous argument gives this bound also in the case $\Aut_{\overline{K}}(E)\cong \Z/4\Z$. 

It remains to address the cases in which $\Aut_{\overline{K}}(E)\cong \Z/2\Z$ and $\#\psi^{-1}(E)\neq r(G)$. As $\mcX_G\cong \mcP(\w')$, by Theorem \ref{thm:WProjMorphisms} the morphism $\psi:\mcX_G\to \mcX_{\GL_2(\Z)}$ that forgets the level  structure can be described by weighted homogeneous polynomials. This gives a parametrization of elliptic curves with $G$-level, away from elliptic curves with extra automorphisms. However, we have already dealt with those elliptic curves with extra automorphisms in the previous paragraph. Thus, by Hilbert's Irreducibility Theorem, the set of points of $\mcX_G$ corresponding to elliptic curves which admit a $G'$-level structure for some proper subgroup $G'<G$ is thin. Then (\ref{eq:admit_level_structure}) follows from Theorem \ref{thm:WProjFin} by taking $\Omega_\infty$ to be the Zariski thin set above (so that $\kappa=0$).
\end{proof}

Let $\mcE_K(B)$ denote the set of isomorphism classes of elliptic curves over $K$ with height less than $B$. In the special case that $G=\GL_2(\Z)$ we obtain the following result using Corollary \ref{cor:WPcount-f=id}.

\begin{corollary}\label{cor:EllipticCount}
Letting 
\[
\kappa=\frac{h_K (2^{r_1+r_2}\pi^{r_2})^{n+1} R_K 10^{r_1+r_2-1}\gcd(2,\varpi_K)}{\varpi_K |\Delta_K| \zeta_K(10)},
\]
we have
\[
\#\mcE_K(B)=\kappa B^{5/6}+O(B^{5/6-1/3d}).
\]
\end{corollary}

For more special cases of Theorem \ref{thm:EllipticCount} see Table \ref{tab:EllipticCount}.

\begin{remark}
Theorem \ref{thm:WProjFin} may also be used to count elliptic curves with prescribed local conditions and level structure; giving results along the lines of \cite[Theorem 3.7 and Corollary 3.10]{CJ23b},  \cite[Theorem 1.4]{CJ23a}, and \cite[Theorem 1.1.2]{Phi25}. This is needed for applications to bounding average ranks (see \cite{CJP24+} for more on this).
\end{remark}

We now prove our main result for counting elliptic curves parametrized by a modular curve isomorphic to $\mcP(2,2)$.

\begin{reptheorem}[\ref{thm:EllipticTwistCount}]
Let $G$ be a subgroup of $\GL_2(\Z/N\Z)$ and let $K$ be a number field of degree $d$ containing $K_G$. Suppose that over $K_G$, the modular curve $\mcX_G$ is isomorphic to $\mcP(2,2)$. 
Define 
\begin{align*}
e'(G)\defeq \frac{[\SL_2(\Z):\Gamma_{G}]}{12}.
\end{align*}
Let $\kappa$ denote the constant $\kappa_f$ in Corollary \ref{cor:WProjTwisted}.
Then the number of isomorphism classes of elliptic curves over $K$ admitting a $G$-level structure and which have naive height less than $B$ is asymptotic to
\[
\begin{cases}
\left(\frac{\kappa}{6}\sum\limits_{\substack{\b\subseteq \O_K\\ \text{square-free}}}\lambda(\b) \sigma(\b,2)\right) B^{1/6}\log(B) & \text{ if } e'(G)=1,\\
\left(\frac{\kappa}{e'(G)-1}\sum\limits_{\substack{\b\subseteq \O_K\\ \text{square-free}}} \lambda(\b) \sigma(\b,2)\right) B^{1/6} & \text{ if } e'(G)>1.
\end{cases}
\]
\end{reptheorem}

\begin{proof}
Let $\psi: \mcX_G\to \mcX(1)$   be the morphism that forgets the level structure, and let $\pi$ be the degree two cover
\begin{align*}
	\pi: \mcP(1,1)&\to \mcP(2,2)\\
	[x:y] &\mapsto [x^2:y^2].
\end{align*}
Let $\varphi_f=\psi\circ \pi$. Then we have the following commutative diagram:
\begin{center}
\begin{tikzcd}
\mcP(1,1) \arrow[d, "\pi"'] \arrow[rd, "\varphi_f"]  \\
 \mcX_G\cong \mcP(2,2) \arrow[r, "\psi"]  & \mcX(1)\cong \mcP(4,6).
\end{tikzcd}
\end{center}
We now compute the degree $\varphi_f$ in two ways. 
As $\varphi_f$ is a morphism of weighted projective stacks, its degree is
\begin{align}\label{eq:deg-wps}
	\deg(\varphi_f)=e(\varphi_f) \frac{\lcm(4,6)}{\lcm(1,1)}=24 \cdot e(\varphi_f).
\end{align}
On the other hand, 
\begin{align}\label{eq:deg-modular}
	\deg(\varphi_f)=\deg(\pi)\cdot \deg(\psi)=2\cdot [\GL_2(\Z/N\Z): G]=2\cdot[\SL_2(\Z): \Gamma_G].
\end{align}
Therefore, $e(\varphi_f)=[\SL_2(\Z): \Gamma_G]/12=e'(G)$.

The desired result then follows from Corollary \ref{cor:WProjTwisted}, taking into account that the naive height on $\mcX_{\GL_2(\Z)}$ is the twelfth power of the height $\Ht_{(4,6)}$ on $\mcP(4,6)$.
\end{proof}

In Theorem \ref{thm:EllipticTwistCount}, the condition that the modular curve $\mcX_G$ is isomorphic to $\mcP(2,2)$ may seem restrictive. However, it turns out that if $\mcX_G$ is a $\mu_2$-gerbe over $\P^1$, then it must be isomorphic to $\mcP(2,2)$. We explain this below, in Proposition \ref{prop:mu_2-gerbe_K-rational} and Proposiiton \ref{prop:mu_2-gerbe_non-trivial}.

\begin{proposition}\label{prop:mu_2-gerbe_K-rational}
	Up to isomorphism, the only $\mu_2$-gerbes over $\P^1_K$ with a $K$-rational point are $\P^1\times \mcP(2)$ and $\mcP(2,2)$. 
\end{proposition}

\begin{proof}
Isomorphism classes of $\mu_2$-gerbes over $\P^1_{K}$ are classified by elements of the \'{e}tale cohomology group 
\begin{equation}
H^2(\P^1_K, \mu_2)\cong H^2(\P^1_{\overline{K}}, \mu_2) \oplus H^2(K,\mu_2) \cong \Z/2\Z \oplus \Br(K)[2],
\end{equation}
where $H^2(\P^1_{\overline{K}}, \mu_2)\cong \Pic(\P^1_{\overline{K}})/2\Pic(\P^1_{\overline{K}})\cong \Z/2\Z$ is generated by the class of $\mathcal{O}(1)$, and where
$\Br(K)[2]$ is the $2$-torsion in the Brauer group of $K$. Let \[
\alpha=(n,\beta)\in H^2(\P^1_K, \mu_2)\cong  \Z/2\Z \oplus \Br(K)[2],
\]
 and let 
$\pi:\mcG\to \P^1_K$ be the $\mu_2$-gerbe over $\P^1_K$ corresponding to $\alpha$. The condition that $\mcG$ has a $K$ rational point is equivalent to the existence of a section $s:\Spec(K)\to \mcG$ of $\pi$ restricted to the fiber $\mcG_{\pi(s)}$. In terms of cohomology, the existence of such a section is equivalent to the vanishing of the pullback class $(\pi\circ s)^{\ast}(\alpha)\in \Br(K)[2]$. We evaluate the pullback on the two components of $\alpha=(n,\beta)$:
\begin{itemize}
	\item The map $(\pi\circ s)^\ast:\Pic(\P^1_K)\to \Pic(K)$ is the zero map because the Picard group of any field vanishes, $\Pic(K)=0$.
	\item The map $(\pi\circ s)^\ast: \Br(K)[2]\to \Br(K)[2]$ is the identity map because $\pi\circ s:\Spec(K) \to \P^1_K$ is a section of the structure morphism $\P^1_K\to \Spec(K)$. 
\end{itemize}
Therefore $(\pi\circ s)^\ast(\alpha)=\beta$, which vanishes exactly when $\beta=0$. This leaves the two cases $\alpha\in \{(0,0), (1,0)\}$. 
\begin{itemize}
\item If $\alpha=(0,0)$, then $\mcG$ is the trivial $\mu_2$-gerbe over $\P^1_K$, which is isomorphic to $\P^1_K\times \mcP(2)$.
\item If $\alpha=(1,0)$, then $\mcG$ is the gerbe of square roots of $\mathcal{O}(1)$, which is isomorphic to the weighted projective stack $\mcP(2,2)$.
\end{itemize}
\end{proof}

The next proposition shows that any modular curve with generic $\mu_2$ stabilizer is a non-trivial $\mu_2$ gerbe.

\begin{proposition}\label{prop:mu_2-gerbe_non-trivial}
	Let $G\subseteq \GL_2(Z/N\Z)$ be a subgroup containing the negative identity matrix, and let $\mcC\defeq\mcX_G\mathbin{\!\!\pmb{\fatslash}} \mu_2$ denote the rigidification\footnote{See, e.g., \cite[Appendix C]{AGV08} for details about rigidification.} along the generic $\mu_2$ stabilizer. Then $\mcX_G$ is not isomorphic to $\mcC\times \mcP(2)$.
\end{proposition}

\begin{proof}
		For contradiction, suppose that $\mcX_G$ is isomorphic to $\mcC\times \mcP(2)$, the trivial $\mu_2$-gerbe over $\mcC$. Then there is a section 
	\begin{equation}
		s: \mcC \to \mcX_G
	\end{equation}
splitting the gerbe. Let $\mcE$ denote the stacky, universal, elliptic curve over $\mcX_G$. Pulling back along the section $s$ gives a family of elliptic curves $E\defeq s^\ast \mcE\to \mcC$, with the property that the pullback of $E$ along the gerbe map $\varphi:\mcX_G\to \mcC$ recovers $\mcE$, i.e., $\varphi^\ast E = \mcE$. However, for a geometric point $x\in \mcX$ corresponding to an elliptic curve $E_x$ with $G$-level structure, the stabilizer $\mu_2$ acts non-trivially on $E_x$ via the usual elliptic involution. It follows that the universal elliptic curve $\mcE$ cannot descend along $\varphi$, contradicting that $\mcX_G$ is isomorphic to $\mcC\times \mcP(2)$.
\end{proof}

Proposition \ref{prop:mu_2-gerbe_K-rational} and Proposition \ref{prop:mu_2-gerbe_non-trivial} together imply that any modular curve isomorphic to a $\mu_2$ gerbe over $\P^1$ must be isomorphic to $\mcP(2,2)$.

\begin{remark}\label{rem:X0(3)}
As mentioned in the introduction, the modular curve $\mcX_{G_0(3)}$ is isomorphic to the weighted projective stack $\mcP(2,6)$. However, Theorem \ref{thm:EllipticCount} does not apply since 
\[
e(G_0(3))=\frac{2\cdot 6}{24} [\SL_2(\Z): \Gamma_0(3)] = 2,
\]
so that $\gcd(w'_0,e(G_0))=\gcd(2,2)=2\neq 1$. If Theorem \ref{thm:EllipticCount} were to apply, then it would give an asymptotic growth rate of $B^{1/3}$, which does not agree with the growth rate of $B^{1/2}$ given in \cite{PPV20}. This illustrates the necessity of assuming the morphism forgetting the level structure has finite defect.

 Similarly, a naive modification of Theorem \ref{thm:EllipticTwistCount} would allow one to relate the asymptotic in the case of $\mcX_{G_0(3)}$ to the asymptotic in the case $\mcX_{G_1(3)}$ (which we can handle using Theorem \ref{thm:EllipticCount}). However, this also leads to the incorrect growth rate of $B^{1/3}$. 
 
 A key difference between the $\mcX_{G_0(3)}$ case and the cases we have considered in this article is that the leading term of the asymptotic in the $\mcX_{G_0(3)}$ case is only accounting for ellipic curves with $j$-invariant zero. From the view of the modular curve, this means that the count is being dominated by the point of $\mcX_{G_0(3)}$ with $\mu_6$-stabilizer (i.e., the order $3$ elliptic point). In contrast, none of the cases considered in this article have an \textit{accumulating substack}.
\end{remark}


\appendix

\section{Bounding the Sum of Divisors Function Over Number Fields}

In this appendix we give an estimate for the sum of divisors function over an arbitrary number field. We adapt the classical argument (see, e.g., \cite[Theorem 6.28]{LeV77}).

Let $K$ be a number field.
 For $\a\subset \O_K$ an ideal with prime factorization $\a=\p_1^{a_1}\dots \p_j^{a_j}$ define the function
\[
\varphi_K(\a)\defeq \prod_{i=1}^j \#(\O_K/\p^{a_i})^\times,
\]
which is an analog of Euler's totient function.

One can show that $\varphi_K$ satisfies the product formula
\begin{align}\label{eq:PhiProductFormula}
\varphi_K(\a)=N_K(\a)\prod_{\p|\a} \left(1-\frac{1}{N_{K/\Q}(\p)}\right)
\end{align}
(see, e.g., \cite{APR20}).

Define the sum of divisors function of an ideal $\a\subset \O_K$ as follows:
\begin{align*}
\sigma_K(\a)\defeq \sum_{\d|\a} N_{K/\Q}(\d),
\end{align*}
where the sum ranges over all divisors $\d$ of $\a$.


\begin{proposition}
Let $\a=\prod_{i=1}^j \p_i^{a_i}$. The function $\sigma_K$ satisfies the following product formula:
\[
\sigma_K(\a)=N_{K/\Q}(\a)\prod_{i=1}^j \frac{1-N_{K/\Q}(\p_i)^{-(a_i+1)}}{1-N_{K/\Q}(\p_i)^{-1}}.
\]
\end{proposition}

\begin{proof}
First we show that $\sigma_K$ is multiplicative. Let $\a_1,\a_2\subset \O_K$ be relatively prime ideals (i.e., if $\p|\a_1$ then $\p\nmid\a_2$). From the multiplicativity of the norm $N_{K/\Q}$, we have that
\begin{align*}
\sigma_K(\a_1\a_2)&=\sum_{\d|\a_1\a_2}N_{K/\Q}(\d)\\
&=\sum_{\substack{\d_1|\a_1 \\ \d_2|\a_2}} N_{K/\Q}(\d_1\d_2)\\
&=\sum_{\d_1|\a_1} N_{K/\Q}(\d_1)\sum_{\d_2|\a_2} N_{K/\Q}(\d_2)\\
&=\sigma_K(\a_1)\cdot\sigma_K(\a_2).
\end{align*}
In particular, 
\[
\sigma_K(\a)=\sum_{i=1}^j \sigma_K(\p_i^{a_i}).
\]
We see that
\begin{align*}
\sigma_K(\p^a)&=\sum_{\d|\p^a} N_{K/\Q}(\d)\\
&=\sum_{i=0}^{a} N_{K/\Q}(\p)^{i}\\
&=\frac{N_{K/\Q}(\p)^{a+1}-1}{N_{K/\Q}(\p)-1}\\
&=N_{K/\Q}(\p)^{a}\frac{1-N_{K/\Q}(\p)^{-(a+1)}}{1-N_{K/\Q}(\p)^{-1}}.
\end{align*}
The desired result now follows by summing over the prime ideals dividing $\a$.
\end{proof}

The following is a generalization, due to Rosen\cite[Lemma 1.1]{Ros99}, of a (weak version) of a result of Mertens \cite{Mer74}.

\begin{lemma}\label{lem:KMertens}
Let $K$ be a number field and let $B\in \R_{>0}$ be a positive constant. Then, as a function of $B$,
\[
\sum_{N_{K/\Q}(\p)\leq B} \frac{1}{N_{K/\Q}(\p)}=\log\log(B)+O(1).
\] 
\end{lemma}

We now give a lower bound for $\varphi_K(\a)$.

\begin{lemma}\label{lem:PhiLowerBound}
For $K$ a number field we have 
\[
\varphi_K(\a)\gg \frac{N_{K/\Q}(\a)}{\log\log(N_{K/\Q}(\a))}.
\]
\end{lemma}

\begin{proof}
From the product formula (\ref{eq:PhiProductFormula}) for $\varphi_K$, we have that
\[
\frac{\varphi_K(\a)}{N_{K/\Q}(\a)}=\prod_{\p|\a} \left(1-\frac{1}{N_{K/\Q}(\p)}\right).
\] 
Thus
\begin{align*}
\log\left(\frac{\varphi_K(\p)}{N_{K/\Q}(\p)}\right)&=\sum_{\p|\a} \log\left(1-\frac{1}{N_{K/\Q}(\p))}\right)\\
&=-\sum_{\p|\a} \frac{1}{N_{K/\Q}(\p)} + \sum_{\p|\a}\left(\frac{1}{N_{K/\Q}(\p)}+\log\left(1-\frac{1}{N_{K/\Q}(\p)}\right)\right).
\end{align*}
For the second sum we have
\begin{align*}
\sum_{\p|\a}\left(\frac{1}{N_{K/\Q}(\p)}+\log\left(1-\frac{1}{N_{K/\Q}(\p)}\right)\right)
&>\sum_{\p|\a}\left(\frac{1}{N_{K/\Q}(\p)}-\frac{1}{N_{K/\Q}(\p)-1}\right)\\
&=-\sum_{\p|\a} \frac{1}{N_{K/\Q}(\p)^2-N_{K/\Q}(\p)}\\
& >-\sum_{N_{K/\Q}(\p)\leq N_{K/\Q}(\a)}\frac{1}{N_{K/\Q}(\p)^2-N_{K/\Q}(\p)}.
\end{align*}
Note that this final sum converges an $N_{K/\Q}(\a)\to\infty$, and therefore 
\[
\log\left(\frac{\varphi_K(\a)}{N_{K/\Q}(\a)}\right)>-\sum_{\p|\a} \frac{1}{N_{K/\Q}(\p)}+O(1).
\]

Partition the set of primes $\p$ dividing $\a$ into the sets
\[
S_1\defeq \{\p|\a : N_{K/\Q}(\p)<\log(N_{K/\Q}(\a))\}
\]
and
\[
S_2\defeq \{\p|\a : N_{K/\Q}(\p)\geq\log(N_{K/\Q}(\a))\}
\]
so that
\[
\sum_{\p|\a} \frac{1}{N_{K/\Q}(\p)}=\sum_{\p\in S_1} \frac{1}{N_{K/\Q}(\p)}+\sum_{\p\in S_2} \frac{1}{N_{K/\Q}(\p)}.
\]
Using Lemma \ref{lem:KMertens} (and the fact that $N_{K/\Q}(\p)\leq \log(N_{K/\Q}(\a))$ for all $\p\in S_1$) we have 
\[
\sum_{\p\in S_1}\frac{1}{N_{K/\Q}(\p)}<\log\log\log(N_{K/\Q}(\a))+O(1).
\]
Since $\log(N_{K/\Q}(\a))\leq N_{K/\Q}(\p)$ for each $\p\in S_2$, we find that
\[
\log(N_{K/\Q}(\a))^{\# S_2}\leq \prod_{\p\in S_2} N_{K/\Q}(\p)\leq N_{K/\Q}(\a).
\]
Taking logarithms and isolating $\# S_2$, we obtain the bound
\[
\# S_2\leq \frac{\log(N_{K/\Q}(\a))}{\log\log(N_{K/\Q}(\a))}.
\]
Therefore
\[
\sum_{\p\in S_2} \frac{1}{N_{K/\Q}(\p)}\leq \sum_{\p\in S_2} \frac{1}{\log(N_{K/\Q}(\a))}=\frac{\# S_2}{\log(N_{K/\Q}(\a))}\leq \frac{1}{\log\log(N_{K/\Q}(\a))}=o(1).
\]
It now follows that
\[
\log\left(\frac{\varphi_K(\a)}{N_{K/\Q}(\a)}\right)>-\sum_{\p|\a} \frac{1}{N_{K/\Q}(\a)}+O(1)\geq -\log\log\log(N_{K/\Q}(\a))+O(1),
\]
and thus
\[
\frac{\varphi_K(\a)}{N_{K/\Q}(\a)}\gg \frac{1}{\log\log(N_{K/\Q}(\a))}.
\]
\end{proof}

We are now ready to give an upper bound for he sum of divisors function.

\begin{theorem}\label{thm:bound-sum-of-divisors}
For $K$ a number field we have that
\[
\sigma_K(\a)\ll N_{K/\Q}(\a)\log\log(N_{K/\Q}(\a)).
\]
\end{theorem}

\begin{proof}
Let $\a=\prod_{i=1}^j \p_i^{a_i}$. Using the product formulas for $\sigma_K$ and $\varphi_K$, we find that
\begin{align*}
\sigma_K(\a)\cdot \varphi_K(\a)&=N_{K/\Q}(\a)\prod_{i=1}^j \frac{1-N_{K/\Q}(\p_i)^{-(a_i+1)}}{1-N_{K/\Q}(\p_i)^{-1}} N_{K/\Q}(\a)\prod_{i=1}^j \left(1-\frac{1}{N_{K/\Q}(\p_i)}\right)\\
&=N_{K/\Q}(\a)^2 \prod_{i=1}^j \left(1-\frac{1}{N_{K/\Q}(\p_i)^{a_i+1}}\right)\\
&<N_{K/\Q}(\a)^2.
\end{align*}
Therefore $\sigma_K(\a)/N_{K/\Q}(\a)<N_{K/\Q}(\a)/\varphi_K(\a)$, so that by Lemma \ref{lem:PhiLowerBound} we have
\[
\frac{\sigma_K(\a)}{N_{K/\Q}(\a)}\ll \log\log(N_{K/\Q}(\a)).
\]
\end{proof}

\bibliographystyle{alpha}
\bibliography{bibfile}

\end{document}